\documentclass[12 pt]{article}
\usepackage[pdftex]{graphicx}
\usepackage{amsmath}
\usepackage{mathtools}
\usepackage{amsthm}
\usepackage{amsfonts}
\usepackage{amssymb}
\usepackage[margin=1.0in]{geometry}
\usepackage{tikz-cd}
\usetikzlibrary{cd}
\usepackage{enumitem}

\newtheorem{theorem}{Theorem}[section]
\newtheorem{cor}[theorem]{Corollary}
\newtheorem{lemma}[theorem]{Lemma}
\newtheorem{prop}[theorem]{Proposition}

\newtheorem{innercustomthm}{Theorem}
\newenvironment{customthm}[1]
{\renewcommand\theinnercustomthm{#1}\innercustomthm}
{\endinnercustomthm}

\newtheorem*{theorem*}{Theorem}

\theoremstyle{definition}
\newtheorem{defin}[theorem]{Definition}
\newtheorem{fact}[theorem]{Fact}
\newtheorem{exa}[theorem]{Example}

\theoremstyle{remark}
\newtheorem*{rem}{Remark}

\newcommand{\flim}[1]{\mathrm{Flim}(#1)}
\newcommand{\age}[1]{\mathrm{Age}(#1)}

\newcommand{\fr}{Fra\"iss\'e }
\newcommand{\frcomma}{Fra\"iss\'e, }
\renewcommand{\phi}{\varphi}

\newcommand{\emb}{\mathrm{Emb}}
\newcommand{\aut}[1]{\mathrm{Aut}(#1)}

\newcommand{\dom}[1]{\mathrm{dom}(#1)}

\newcommand{\im}[1]{\mathrm{Im}(#1)}

\renewcommand{\b}[1]{\mathbf{#1}}
\newcommand{\bb}[1]{\mathbb{#1}}
\renewcommand{\c}[1]{\mathcal{#1}}

\newcommand{\crit}{\mathrm{Crit}}
\newcommand{\Left}{\mathrm{Left}}
\newcommand{\Path}{\mathrm{Path}}
\renewcommand{\sp}{\mathrm{Sp}}
\newcommand{\ac}{\mathrm{AC}}

\renewcommand{\succ}{\mathrm{Succ}}
\newcommand{\is}{\mathrm{IS}}

\newcommand{\Int}{\mathrm{Int}}
\newcommand{\Irr}{\mathrm{Irr}}
\newcommand{\ct}{\mathrm{CT}}
\newcommand{\aemb}{\mathrm{AEmb}}
\newcommand{\oemb}{\mathrm{OEmb}}
\def\-{\raisebox{.30pt}{-}}

\begin{document}
	\title{On big Ramsey degrees for binary free amalgamation classes}
	\author{Andy Zucker}
	\maketitle
	
\begin{abstract}
	Generalizing and simplifying recent work of Dobrinen, we show that if $\c{L}$ is a finite binary relational language and $\c{F}$ is a finite set of finite irreducible $\c{L}$-structures, then the class $\c{K} = \mathrm{Forb}(\c{F})$ has finite big Ramsey degrees.
	\let\thefootnote\relax
	\let\thefootnote\relax\footnote{The author was supported by NSF Grants \ DMS 1803489 and DMS  and the ANR project AGRUME (ANR-17-CE40-0026).}
\end{abstract}
	
The infinite Ramsey theorem \cite{Ram} states that for any $k, r< \omega$ and any coloring $\chi\colon [\bb{N}]^k\to r$, there is some infinite $S\subseteq \bb{N}$ with $\big|\chi\big[[S]^k\big]\big| = 1$, i.e.\ that $\chi$ is monochromatic on $k$-tuples from $S$. The situation becomes more interesting when we place restrictions on the set $S$ from the theorem. For instance, suppose $\chi\colon [\bb{Q}]^k\to r$ is a coloring. Can we find an infinite set $S\subseteq \bb{Q}$ which is \emph{order-isomorphic to $\bb{Q}$} and with $\big|\chi\big[[S]^n\big]\big| = 1$? The answer is no; let $\bb{Q} = \{q_n: n< \omega\}$ be an enumeration. We use this enumeration to define the following $2$-coloring of pairs: given $m< n< \omega$, set $\chi(\{q_m, q_n\}) = 0$ if $q_m < q_n$, and set $\chi(\{q_m, q_n\}) = 1$ if $q_n < q_m$. Then whenever $S\subseteq \bb{Q}$ is order isomorphic to $\bb{Q}$, we can find pairs from $S$ of each color. 

Remarkably, Galvin \cite{Gal} shows that if $r < \omega$ and $\chi\colon [\bb{Q}]^2\to r$ is a coloring, there is some $S\subseteq \bb{Q}$ which is order isomorphic to $\bb{Q}$ and with $\big|\chi\big[[S]^2\big]\big|\leq 2$. This was extended by Devlin \cite{Dev}, who showed that for every $k< \omega$, there is $t_k< \omega$ so that whenever $r< \omega$ and $\chi\colon [\bb{Q}]^k\to r$ is a coloring, there is some $S\subseteq \bb{Q}$ which is order isomorphic to $\bb{Q}$ and with $\big|\chi\big[[S]^k\big]\big|\leq t_k$. While Devlin mentions that the existence of $t_k$ was known to Laver, Devlin gives the optimal value of $t_k$; it is the $k^{th}$ odd tangent number. The first few odd tangent numbers are $t_2 = 2$, $t_3 = 16$, $t_4 = 272$. So for instance, there is a $16$-coloring of triples of rationals so that every $S\subseteq \bb{Q}$ which is order-isomorphic to $\bb{Q}$ sees all $16$ colors, whereas given any $17$-coloring $\chi$ of triples of rationals, there is some $S\subseteq \bb{Q}$ order-isomorphic to $\bb{Q}$ so that $\big|\chi\big[[S]^3\big]\big|\leq 16$.

In order to prove Devlin's theorem, one first identifies the rationals with the infinite rooted binary tree $2^{< \omega}$. Then one uses a Ramsey theoretic result on trees, and transfers this back to Ramsey theoretic information about the rationals. The tree Ramsey theorem used here is Milliken's theorem \cite{Mil}, the proof of which is not easy. To prove it, one first proves the Halpern--L\"auchli theorem \cite{HL} and repeatedly uses this theorem to complete the inductive step in Milliken's theorem. The proof of the Halpern--L\"auchli theorem is also not easy, and there are several known proofs. Direct, ``combinatorial" proofs do exist, but we draw attention to a proof due to Harrington which uses some machinery from set theory, namely techniques from forcing and the Erd\H{o}s-Rado theorem. Versions of this proof can be found in \cite{FarTod} and \cite{DobForcing}. We remark that though the arguments use some of the forcing formalism, the key object is the forcing poset; the use of this larger object to prove facts about smaller objects is analogous to proving facts about $\bb{N}$, such as van der Waerden's theorem, using the larger space $\beta \bb{N}$ of ultrafilters on $\bb{N}$.  

A natural setting to study combinatorial problems of this form is that of a \emph{\fr structure}. Suppose $\c{L}$ is a countable relational language and that $\b{K}$ is a countably infinite $\c{L}$-structure. We say that $\b{K}$ is \emph{\fr}if whenever $\b{A}\subseteq \b{K}$ is finite and $f\colon \b{A}\to \b{K}$ is an embedding, then there is an automorphism $g$ of $\b{K}$ with $g|_{\b{A}} = f$. Recall that $\age{\b{K}} := \{\b{A} \text{ finite}: \emb(\b{A}, \b{K})\neq \emptyset\}$. The class $\age{\b{K}}$ is always hereditary, and if $\b{K}$ is \frcomma it also satisfies the \emph{amalgamation property}: for any $\b{A}, \b{B}, \b{C}\in \age{\b{K}}$ and embeddings $f\colon \b{A}\to \b{B}$ and $g\colon \b{A}\to \b{C}$, there is $\b{D}\in \age{\b{K}}$ and embeddings $r\colon \b{B}\to \b{D}$ and $s\colon \b{C}\to \b{D}$ with $r\circ f = s\circ g$. Conversely, given any hereditary class $\c{K}$ of finite $\c{L}$-structures with the amalgamation property, then modulo some non-triviality assumptions on $\c{K}$, there is a \fr structure $\b{K}$ with $\age{\b{K}} = \c{K}$. This \fr structure is unique up to isomoprhism, and is denoted $\flim{\c{K}}$, the \emph{\fr limit} of $\c{K}$. 

Now consider a \fr structure $\b{K} = \flim{\c{K}}$ and some $\b{A}\in \c{K}$. Is there $t_\b{A}< \omega$ so that for any coloring $\chi\colon \emb(\b{A}, \b{K})\to r$ for some $r< \omega$, there is $\eta\in \emb(\b{K}, \b{K})$ for which $|\im{\chi\circ \eta}|\leq t_\b{A}$. The least number $t_\b{A}$ with this property, if it exists, is called the \emph{big Ramsey degree} of $\b{A}$ in $\c{K}$, and we say that $\c{K}$ has \emph{finite big Ramsey degrees} if every $\b{A}\in \c{K}$ has some finite big Ramsey degree. As an example, if $\c{K}$ is the class of finite linear orders, then $\flim{\c{K}} = \bb{Q}$, and Devlin's result says that $\c{K}$ has finite big Ramsey degrees. Until recently, not many \fr classes were known to have finite big Ramsey degrees; some examples are the classes of finite sets \cite{Ram}, finite graphs \cite{Sauer2}, finite-distance ultrametric spaces \cite{NVT}, and finite linear orders with a labeled partition \cite{LNVTS}. In all of these examples, one codes the structure at hand as a tree and directly uses Ramsey's theorem, Milliken's theorem, or a soft variant thereof.

This situation should be contrasted with what is known about \emph{small} Ramsey degrees (often, the word ``small" is just omitted). Here, instead of searching for $\eta\in \emb(\b{K}, \b{K})$ for which $\chi\circ \eta$ takes few colors, we weaken our demand to only ask for $\eta_\b{B}\in \emb(\b{B}, \b{K})$ for arbitrarily large, but finite $\b{B}\in \c{K}$. Unlike the situation with big Ramsey degrees, where few examples are known, there are many general theorems which give \fr classes with finite Ramsey degrees. An early major result is the theorem of Ne\v{s}et\v{r}il and R\"odl \cite{NR}, who show that every free amalgamation class (to be discussed in more detail soon) in a finite relational language has finite Ramsey degrees. Beyond free amalgamation classes, examples worth mentioning include the class of finite posets \cite{NRPosets} and the class of finite rational-distance metric spaces \cite{N}. More recently, Hubi\v{c}ka and Ne\v{s}et\v{r}il \cite{HN} prove an extremely general theorem giving very mild sufficient conditions for a class to have finite Ramsey degrees; in particular, their framework generalizes every example mentioned in this paragraph.

In a recent, highly technical series of papers, Dobrinen \cite{Dob0, Dob} has shown that for each $n< \omega$, the class of finite graphs which do not embed an $n$-clique has finite big Ramsey degrees. Indeed, a major motivation of the work here was to understand, generalize, and simplify the proof of this result. Even in the simplest case $n = 3$, it is far less simple to code triangle-free graphs using trees. Dobrinen introduces the notion of a ``coding tree" and proves analogs of the Halpern--L\"auchli and Milliken theorems for these objects; the proof of the Halpern--L\"auchli analog is inspired by Harrington's proof and uses forcing and the Erd\H{o}s-Rado theorem.

Recall that if $\c{L}$ is a relational language, an $\c{L}$-structure $\b{F}$ is called \emph{irreducible} if every $a\neq b\in \b{F}$ is contained in some non-trivial relation. If $\c{F}$ is some set of finite irreducible $\c{L}$-structures, then $\mathrm{Forb}(\c{F})$ denotes the class of finite $\c{L}$-structures which do not embed any member of $\c{F}$. We can assume that the members of $\c{F}$ are pair-wise non-embeddable. The class $\mathrm{Forb}(\c{F})$ is always a \fr class, and in fact satisfies a stronger form of the amalgamation property called \emph{free amalgamation}. This means that given an amalgamation problem $f\colon \b{A}\to \b{B}$ and $g\colon \b{A}\to \b{C}$, we can find a solution $r\colon \b{B}\to \b{D}$ and $s\colon \b{C}\to \b{D}$ satisfying the following.
\begin{enumerate}
	\item 
	$\b{D} = \im{r}\cup \im{s}$.
	\item 
	$\im{r}\cap \im{s} = \im{r\circ f} = \im{g\circ s}$.
	\item 
	Whenever $a\neq b\in \b{D}$ is contained in a non-trivial relation, we either have $\{a, b\}\subseteq \im{r}$ or $\{a, b\}\subseteq \im{s}$.
\end{enumerate}
Conversely, every \fr free amalgamation class is of the form $\mathrm{Forb}(\c{F})$ for some set of finite irreducible $\c{L}$-structures. We remark that $\c{F}$ can be infinite. Our main theorem is the following, which can be thought of as the binary case of an ``infinitary Ne\v{s}et\v{r}il--R\"odl theorem."
\vspace{2 mm}

\begin{theorem*}
	\label{Thm:Main}
	Let $\c{L}$ be a finite binary relational language, and suppose $\c{K} = \mathrm{Forb}(\c{F})$ for a finite set $\c{F}$ of finite, irreducible $\c{L}$-structures. Then $\c{K}$ has finite big Ramsey degrees.
\end{theorem*}
\vspace{2 mm}

The restriction that $\c{F}$ be finite is essential in our proof; it is possible that the conclusion of the theorem fails whenever $\c{F}$ is infinite, and hopefully this can be investigated in future work. Indeed, Sauer \cite{Sauer} has shown that some free amalgamation classes of directed graphs do not have finite big Ramsey degrees. 

The structure of our proof is broadly similar to Dobrinen's. In section \ref{Sec:CodeTree}, we define a notion of coding tree, and our Theorems \ref{Thm:HL} and \ref{Thm:Milliken} are the analogues of the Halpern-L\"{a}uchli and Milliken theorems that we need. The key difference between \cite{Dob0} and \cite{Dob} and what we do here is that our notion of coding tree is greatly simplified. Whereas the ``strong coding trees'' in \cite{Dob0} and \cite{Dob} have levels dedicated individually to splitting and coding, our trees always split and always code. In fact, the underlying tree is just $k^{< \omega}$ for some $k< \omega$, and the ``coding tree" is simply the assignment of one coding node and one unary predicate to each level. This in turn reduces the complexity of the forcing argument.  It would be interesting to try to remove the forcing entirely.

Another motivation of this work was to attempt to understand the \emph{exact} big Ramsey degrees for the classes of $n$-clique-free graphs. This would allow for a potential application to topological dynamics as developed in \cite{Zuc}, analogous to the dynamical interpretation of small Ramsey degree results developed by Kechris, Pestov, and Todor\v{c}evi\'c \cite{KPT}. However, we do not attempt to undertake this at this time. The determination of the exact big Ramsey degrees for all classes from the main theorem will appear in a forthcoming joint work with Balko, Chodounsk\'y, Dobrinen, Hubi\v{c}ka, Kone\v{c}n\'y, and Vena.

The paper is organized as follows. Section~\ref{Sec:EnumStruc} introduces Ramsey degrees in the context of enumerated structures, setting up several conventions we use throughout the paper. Section~\ref{Sec:CodeTree} defines our coding trees and the ``aged embeddings" between them. Section~\ref{Sec:RamseyAE} proves our analogs of the Halpern--L\"auchli and Milliken theorems, regarding colorings of the aged embeddings of one coding tree into another. The final two sections undertake the work of transferring the Ramsey theorem for coding trees back into a Ramsey theorem about structures; Section~\ref{Sec:Envelopes} more abstractly, and Section~\ref{Sec:RamseyStructure} in the specific situation pertaining to the main theorem.

Section~\ref{Sec:RamseyAE} is difficult; on a first reading, the reader might choose to skip it, only reading the statement of Theorem~\ref{Thm:Milliken}. An appendix gives a self-contained and fairly ``combinatorial" introduction to the ideas from forcing needed to understand the proof of Theorem~\ref{Thm:HL}.

\subsection*{Notation}

Our notation is mostly standard. We identify natural numbers $n< \omega$ with their set of predecessors. If $m< n< \omega$, we sometimes write $[m, n] := \{k < \omega: m\leq k\leq n\}$. If $f$ is a function and $S\subseteq \dom{f}$, we write $f[S] := \{f(s): s\in S\}$. Sometimes function composition is simply denoted by $g\cdot f$ rather than $g\circ f$; if $g$ is a function and $F$ is a set of functions with range $\dom{g}$, we write $g\cdot F := \{g\cdot f: f\in F\}$.
	
\section{Enumerated structures and Ramsey degrees}
\label{Sec:EnumStruc}

\begin{defin}
	\label{Def:RamseyDegree}
	Suppose that $\b{K}$ is an infinite first-order structure, that $\b{A}$ is a finite structure with $\emb(\b{A}, \b{K})\neq \emptyset$, and that $\ell < r < \omega$. We write 
	$$\b{K}\to (\b{K})^{\b{A}}_{r, \ell}$$
	if for any function $\chi\colon \emb(\b{A}, \b{K})\to r$, there is $\eta\in \emb(\b{K}, \b{K})$ so that 
	$$|\im{\chi\cdot \eta}| = \left|\chi[\eta\cdot \emb(\b{A}, \b{K})]\right|\leq \ell.$$ 
	The \emph{Ramsey degree of $\b{A}$ in $\b{K}$}, if it exists, is the least $\ell< \omega$ so that $\b{K}\to (\b{K})^{\b{A}}_{r, \ell}$ holds for every $r> \ell$, or equivalently for $r = \ell+1$.
\end{defin}
\vspace{2 mm}

When $\b{K} = \flim{\c{K}}$ for some \fr class $\c{K}$, one says that $\b{A}\in \c{K}$ has \emph{big Ramsey degree} $\ell$ exactly when $\b{A}$ has Ramsey degree $\ell$ in $\b{K}$. We say that $\c{K}$ has \emph{finite big Ramsey degrees} if every $\b{A}\in \c{K}$ has some finite big Ramsey degree. 

Some authors formulate the above definitions with respect to \emph{copies} rather than embeddings, where a copy is simply the image of an embedding. If $\b{A}$ has Ramsey degree $\ell$ in $\b{K}$ with respect to copies, then one can show that $\b{A}$ has Ramsey degree $\ell\cdot |\aut{\b{A}}|$ with respect to embeddings. So as far as showing that a \fr class has finite big Ramsey degrees, either definition is acceptable.

For the purposes of this paper, a middle ground between these two approaches will be convenient. 
\vspace{2 mm}   

\begin{defin}
	\label{Def:EnumeratedStructure}
	An \emph{enumerated structure} is simply a structure $\b{K}$ whose underlying set is the cardinal $|\b{K}|$. In this note, all enumerated structures will be countable. If $m < |\b{K}|$, we write $\b{K}_m := \b{K}\cap m$. 
	
	We let $<$ denote the usual order on $\omega$. If $\b{A}$ is another enumerated structure, then an \emph{ordered embedding} of $\b{A}$ into $\b{K}$ is simply an embedding of $\langle \b{A}, <\rangle$ into $\langle \b{K}, <\rangle$. We denote the collection of ordered embeddings of $\b{A}$ into $\b{K}$ by $\oemb(\b{A}, \b{K})$.
	
	The \emph{ordered} Ramsey degree of $\b{A}$ in $\b{K}$ is just the Ramsey degree of $\langle \b{A}, <\rangle$ in $\langle \b{K}, <\rangle$. 
\end{defin}
\vspace{2 mm}

Our primary reason for considering enumerated structures is that one can use properties of the enumeration to define unavoidable colorings of $\emb(\b{A}, \b{K})$. For example, keep in mind the very first example about coloring pairs of rationals from the introduction. In fact, what we will show in some sense is that these are the \emph{only} colorings one needs to watch out for. The notion of \emph{coding tree} developed in the next section elaborates on this idea.

For the final proposition of this section, the following notation will be helpful. Suppose $\b{A}$ is a structure, $S$ is a set, and $f\colon S\to \b{A}$ is an injection. Then $\b{A}{\cdot}f$ is the structure on $S$ making $f$ an embedding. If $a_0,...,a_{n-1}\in \b{A}$, then $\b{A}(a_0,...,a_{n-1})$ is the structure on $n$ so that $i\to a_i$ is an embedding. 

We also recall that a \fr class is a \emph{strong amalgamation} class if it satisfies the first two items from the definition of a free amalgamation class (introduction, before the statement of the main theorem).
\vspace{2 mm}

\begin{prop}
	\label{Prop:EnumRamseyDeg}
	Let $\c{K}$ be a \fr class, and suppose $\b{K} = \flim{\c{K}}$ is enumerated. Then if every enumerated $\b{A}\in \c{K}$ has finite ordered Ramsey degree in $\b{K}$, then $\c{K}$ has finite big Ramsey degrees. If $\c{K}$ is a strong amalgamation class, this is iff. 
\end{prop}

\begin{proof}
	First assume that every enumerated $\b{A}\in \c{K}$ has finite ordered Ramsey degree in $\b{K}$. Fix some $\b{A}\in \c{K}$, and write $B = \{\sigma_i: i< n\}$ for the set of bijections from $|\b{A}|$ to $\b{A}$. Then we can write
	$$\emb(\b{A}, \b{K}) = \bigsqcup_{i < n} \oemb(\b{A}{\cdot}\sigma_i, \b{K}).$$
	For each $i < n$, let $\ell_i$ denote the ordered Ramsey degree of $\b{A}{\cdot}\sigma_i$ in $\b{K}$. If $\chi\colon \emb(\b{A}, \b{K})\to r$ is a coloring for some $r< \omega$, we can find $\eta_i\in \oemb(\b{K}, \b{K})$ so that $\oemb(\b{A}{\cdot}\sigma_i, \b{K})$ sees at most $\ell_i$ colors for the coloring $\chi\cdot \eta_0\cdots \eta_{i-1}$. Setting $\eta = \eta_0\circ \cdots \circ \eta_{n-1}$, then $|\chi\cdot \eta\cdot \emb(\b{A}, \b{K})|\leq \sum_{i< n} \ell_i$.
	
	We remark that if $\c{K}$ is a strong amalgamation class, one can show that the Ramsey degree of $\b{A}$ in $\b{K}$ is exactly $\sum_{i< n} \ell_i$.
	
	Now assume that $\c{K}$ is a strong amalgamation class with finite big Ramsey degrees. Let $\b{A}\in \c{K}$ be an enumerated structure with big Ramsey degree $\ell< \omega$, and let $\chi\colon \oemb(\b{A}, \b{K})\to r$ be a coloring for some $r< \omega$. Define a coloring $\xi\colon \emb(\b{A}, \b{K})\to r+1$ extending $\chi$ by setting $\xi(f) = r$ whenever $f\in \emb(\b{A}, \b{K})\setminus \oemb(\b{A}, \b{K})$. Find $\eta\in \emb(\b{K}, \b{K})$ with $|\im{\xi{\cdot}\eta}|\leq \ell$. Now since $\c{K}$ is a strong amalgamation class, we can find $\theta\in \emb(\b{K}, \b{K})$ so that $\eta{\cdot}\theta\in \oemb(\b{K}, \b{K})$. It then follows that $|\im{\chi{\cdot}(\eta{\cdot}\theta)}|\leq \ell$.
\end{proof}
	
\section{Coding trees}
\label{Sec:CodeTree}

Suppose that $\c{L} = \{U_0,...,U_{k-1}; R_0,...,R_{k-1}\}$ is a finite binary relational language, where the $U_i$ are unary and the $R_i$ are binary, throwing in extra symbols to ensure that there are $k$ of each. For notation, if $\b{A}$ is an $\c{L}$-structure and $a, b\in \b{A}$, we let $U^\b{A}(a) = i< k$ iff $i< k$ is unique so that $U^\b{A}_i(a)$ holds, and similarly for $R^\b{A}(a, b)$. This helps to reduce subscripts. By enlarging $\c{L}$ if needed, we may assume that the following all hold whenever $\b{A}$ is an $\c{L}$-structure.
\begin{itemize}
	\item 
	Each $R_i^\b{A}$ is non-reflexive.
	\item 
	$\b{A} = \bigsqcup_{i< k} U_i^\b{A}$.
	\item 
	$\b{A}^2\setminus \{(a, a): a\in \b{A}\} = \bigsqcup_{i< k} R_i^\b{A}$. Here $R_0$ will play the role of ``no relation." 
	\item 
	There is a map $\mathrm{Flip}\colon k\to k$ with $\mathrm{Flip}^2 = \mathrm{id}$ so that for any $a\neq b\in \b{A}$, we have $R^\b{A}(a, b) = i$ iff $R^\b{A}(b, a) = \mathrm{Flip}(i)$. We demand that $\mathrm{Flip}(0) = 0$.
\end{itemize} 
We remark that we almost never need to refer to $\mathrm{Flip}$ explicitly; this just ensures that when we are discussing an $\c{L}$-structure and we know $R^\b{A}(a, b)$, then there is no need to mention $R^\b{A}(b, a)$. As a rule of thumb, when working with an enumerated $\c{L}$-structure $\b{A}$ with $m< n< |\b{A}|$, we will define binary relations going \emph{down} the enumeration, i.e.\ we define $R^\b{A}(n, m)$, and then one knows $R^\b{A}(m, n)$ via $\mathrm{Flip}$. Since $R_0$ plays the role of no relation, this is also the sense in which we understand notions such as ``irreducible structure," ``free amalgam," etc.

We briefly discuss some terminology relating to the tree $k^{<\omega}$. To minimize superscripts, \textbf{we denote this tree by} $T$, and we set $T(n) = k^n$, $T({<}n) = k^{<n}$, etc.; when we do write $k^n$, we will mean as a number. If we wish to write out an element of $T$ explicitly, we use angle brackets, i.e.\ $\langle 0110\rangle\in T(4)$. If $s\in T$, then the \emph{height} of $s$ is the integer $\ell(s)$ with $s\in T(\ell(s))$. If $s, t\in T$, we say that $s$ is an \emph{initial segment} of $t$ or that $t$ \emph{extends} $s$ and write $s\sqsubseteq t$ if $\ell(s)\leq \ell(t)$ and $t|_{\ell(s)} = s$. The \emph{meet} of $s, t\in T$, denoted $s\wedge t$, is the longest common initial segment of $s$ and $t$. If $s\in T$ and $n\geq \ell(s)$, then $\succ(s, n) := \{t\in T(n): t\sqsupseteq s\}$, and $\Left(s, n)\in \succ(s, n)$ is the extension of $s$ to height $n$ by adding zeros to the end. If $S\subseteq T$ satisfies $\max\{\ell(s): s\in S\}\leq n$, then $\succ(S, n) := \bigcup \{\succ(s, n): s\in S\}$ and $\Left(S, n) = \{\Left(s, n): s\in S\}$. The set of \emph{immediate successors} of $s$ is the set $\is(s) := \succ(s, \ell(s)+1)$. If $i< k$, then $s^\frown i\in \is(s)$ denotes the immediate successor of $s$ formed by setting $s^\frown i(\ell(s)) = i$. We set $\is(S) := \bigcup \{\is(s): s\in S\}$. If $m< \omega$, then $\pi_m\colon T({\geq}m)\to T(m)$ will denote the restriction map.
\vspace{2 mm}

\begin{defin}
	\label{Def:CodingTree}
	A \emph{coding tree} $\ct := (c, u)$ is a pair of functions $c\colon n\to T({<}n)$ and $u\colon n\to k$ for some $n\leq \omega$, where $c(m)\in T(m)$ for each $m< n$. We call the image of $c$ the set of \emph{coding nodes} of $c$, and we say that the coding node $c(m)$ has \emph{unary} $i< k$ if $u(m) = i$. 
	
	Given $(c, u)$ a coding tree with domain $n$, we recover an $\c{L}$-structure $\b{X}(c, u)$ on $n$, where given $i < n$, we set $U^{\b{X}(c, u)}(i) = u(i)$, and given $i < j < n$, we declare that $R^{\b{X}(c, u)}(j, i) = c(j)(i)$ holds. Conversely, if $\b{A}$ is an $\c{L}$-structure on $n$, the coding tree $\ct^\b{A} := (c^\b{A}, u^\b{A})$ is defined by setting $u^\b{A}(i) = U^\b{A}(i)$ whenever $i< n$ and $c^\b{A}(j)(i) = R^\b{A}(j, i)$ whenever $i< j < n$.    
\end{defin}
\vspace{2 mm}

Soon, we will fix once and for all a free amalgamation class of $\c{L}$-structures and an enumerated \fr limit. We will want this enumerated \fr limit to have one extra property.
\vspace{2 mm}

\begin{defin}
	\label{Def:LeftDense}
	Suppose $\c{K}$ is a free amalgamation class of $\c{L}$-structures and that $\b{K}$ is an enumerated \fr limit of $\c{K}$. Then we say that $\b{K}$ is \emph{left dense} if for any $m< \omega$, whenever $\b{A}\in \c{K}$ has underlying set $m+1$ and $\b{A}_m = \b{K}_m$, then there is $n\geq m$ with $\b{K}(0,...,m-1, n) = \b{A}$ and with $R(n, r) = 0$ for every $r\in [m, n-1]$. We note that a left-dense \fr limit can easily be constructed inductively. 
	
	In terms of the coding tree $\ct^\b{K}$, this means that whenever $s\in T(m)$ is a node which can be extended to a coding node of unary $i< k$, then there is $n\geq m$ with $\Left(s, n) = c^\b{K}(n)$ and $u(n) = i$.
\end{defin} 

\textbf{We now fix} a \fr free amalgamation class $\c{K}$ and an enumerated \fr limit $\b{K}$ which is left dense. When referring to this fixed \fr limit $\b{K}$, we typically omit various superscripts, writing $\ct = \ct^\b{K} = (c, u)$, $U = U^\b{K}$, and $R = R^\b{K}$.

We now define the class of \emph{aged embeddings} between coding trees. To do this, we first define an \emph{embedding}, and an aged embedding will be an embedding satisfying extra properties. 
\vspace{2 mm}

\begin{defin}
	\label{Def:Embedding}
	Suppose $\b{A}$ is an enumerated structure with $\age{\b{A}}\subseteq \c{K}$. Then a map $f\colon T(<|\b{A}|)\to T$ is an \emph{embedding} of $\ct^\b{A}$ into $\ct$ if the following all hold:
	\begin{enumerate}
		\item 
		$f$ is an injection.
		\item 
		There is an increasing injection $\tilde{f}\colon |\b{A}|\to \omega$ so that $f[T(m)]\subseteq T(\tilde{f}(m))$ for each $m< |\b{A}|$.
		\item 
		$f$ preserves meets.
		\item 
		For each $s\in T(<|\b{A}|)$ and $i< k$, we have $f(s^\frown i) \sqsupseteq f(s)^\frown i$.
		\item 
		For each $m < |\b{A}|$, we have $f(c^\b{A}(m)) = c(\tilde{f}(m))$ and $u^\b{A}(m) = u(\tilde{f}(m))$.
	\end{enumerate}
	We write $\emb(\ct^\b{A}, \ct)$ for the collection of such embeddings. If $f\in \emb(\ct^\b{A}, \ct)$, then the induced map $\tilde{f}\colon |\b{A}|\to \omega$ is in $\emb(\b{A}, \b{K})$.
\end{defin}
\begin{rem}
	If $|\b{A}| = n$ and we delete item $5$, we recover the notion of a \emph{strong similarity} from $T(< n)$ into $T$.
\end{rem}
\vspace{2 mm} 

A major issue we will need to address is as follows. Suppose $\b{A}$ and $\b{B}$ are enumerated $\c{L}$-structures with $\b{A}$ an initial segment of $\b{B}$. Then $\ct^\b{A}$ is also an initial segment of $\ct^\b{B}$. Given $f\in \emb(\ct^\b{A}, \ct)$, does $f$ extend to $\ct^\b{B}$? 
\vspace{2 mm}

\begin{exa}
	\label{Exa:EmbeddingDoesntExtend}
	Let $\c{K}$ be the class of all finite triangle-free graphs. So $k = 2$, and $T$ is just the infinite binary tree. Suppose the enumerated \fr limit $\b{K}$ has $R(1, 0) = 1$, i.e.\ the vertices $1$ and $0$ are adjacent. So $c(0) = \langle\emptyset\rangle$, and $c(1) = \langle 1\rangle$. Note that there cannot be any coding nodes which extend $\langle 11\rangle$, since if $c(n)(0) = c(n)(1) = 1$, then $\{n, 1, 0\}$ is a triangle. 
	
	Suppose $\b{A} = \b{K}_1$ and $\b{B} = \b{K}_2$. Let $f\in \emb(\ct^\b{A}, \ct)$ be given by setting $f(\langle \emptyset\rangle) = \langle 1\rangle$. Then $f$ does not extend to $\ct^\b{B}$, since $f(\langle 1\rangle)$ must be a coding node which extends $\langle 11\rangle$, which is impossible by the discussion above.  
\end{exa}
\vspace{2 mm}

To avoid this problem, we enrich the structure of $\ct$, adding local information at each level describing the possible configurations of coding nodes above. We will then consider the ``aged embeddings" which respect this local information.
\vspace{2 mm} 

\begin{defin}\mbox{}
	\label{Def:AgedSetsEmbeddings}
	\vspace{-3 mm}
	
	\begin{enumerate}
		\item 
		Let $X$ be a finite set. An \emph{$X$-labeled $\c{L}$-structure} is an $\c{L}$-structure $\b{B}$ equipped with a function $\phi\colon \b{B}\to X$. We can view $\phi$ as an expansion of $\b{B}$ by unary predicates coming from $X$. As a rule of thumb, we assume that the underlying set of an $X$-labeled $\c{L}$-structure is disjoint from $\omega$.
		\item 
		An \emph{aged set} is a finite set $X$ equipped with a class $\c{C}_X$ of finite $X$-labeled $\c{L}$-structures. We allow $\c{C}_X = \emptyset$.
		\item 
		Given $(X, \c{C}_X)$ and $(Y, \c{C}_Y)$ aged sets, an \emph{age map} from $X$ to $Y$ is an injection $f\colon X\to Y$ with the property that $(\b{B}, \phi)\in \c{C}_X$ iff $(\b{B}, f\circ \phi)\in \c{C}_Y$. 
		\item 
		Suppose $\b{A}$ is an enumerated $\c{L}$-structure. Let $m\leq |\b{A}|$, $m< \omega$, and suppose $(\b{B}, \phi)$ is a $T(m)$-labeled $\c{L}$-structure, where we assume that $\b{B}\cap \omega = \emptyset$. Then $\b{B}[\phi, \b{A}]$ is the $\c{L}$-structure on $\b{B}\cup \b{A}_m$, with $\b{B}$ and $\b{A}_m$ as induced substructures (thus determining the unary predicates), and for $i< m$ and $b\in \b{B}$, we have $R^{\b{B}[\phi, \b{A}]}(b, i) = \phi(b)(i)$.
		
		When $\b{A} = \b{K}$, we omit it from the notation, simply writing $\b{B}[\phi]$.
		\item 
		Suppose $\b{A}$ is an enumerated structure with $\age{\b{A}}\subseteq \c{K}$. For $m\leq |\b{A}|$, $\c{C}^\b{A}(m)$ is the class of $T(m)$-labeled $\c{L}$-structures $(\b{B}, \phi)$ so that $\b{B}[\phi, \b{A}]\in \c{K}$. If $m$ is understood, we omit it. If $S\subseteq T(m)$, then we set $\c{C}^\b{A}(S) = \{(\b{B}, \phi)\in \c{C}^\b{A}(m): \im{\phi}\subseteq S\}$. 
		
		When $\b{A} = \b{K}$, we omit the superscript, simply writing $\c{C}(m)$ or $\c{C}$. We equip $S\subseteq T(m)$ with the class $\c{C}(S)$ unless explicitly mentioned otherwise. For example, if we refer to an age map $f\colon (T(m), \c{C}^\b{A})\to T(n)$, the range is equipped with the class $\c{C}(n)$.
		\item 
		Suppose $\b{A}$ is an enumerated structure with $\age{\b{A}}\subseteq \c{K}$. Then $f\in \emb(\ct^\b{A}, \ct)$ is an \emph{aged embedding} if for each $m< |\b{A}|$, the map $f\colon (T(m), \c{C}^\b{A})\to T(\tilde{f}(m))$ is an age map. Write $\aemb(\ct^\b{A}, \ct)$ for the set of aged embeddings of $\ct^\b{A}$ into $\ct$.	
	\end{enumerate}
\end{defin}
\vspace{2 mm}

\begin{exa}
	\label{Exa:3FreeAges}
	We revisit Example~\ref{Exa:EmbeddingDoesntExtend}, so once again, $\c{K}$ is the class of finite triangle-free graphs, and $\b{K}$ is an enumerated \fr limit with the property that the vertices $0$ and $1$ are adjacent. Then we can completely describe the classes $\c{C}(0)$, $\c{C}(1)$, and $\c{C}(2)$. The class $\c{C}(0)$ is simply the class of $T(0)$-labeled finite triangle-free graphs; since $|T(0)| = 1$, we can just identify this class with $\c{K}$. The class $\c{C}(1)$ is the class of $T(1)$-labeled finite triangle-free graphs $(\b{B}, \phi)$ such that whenever $a\neq b\in \b{B}$ have $\phi(a) = \phi(b) = \langle 1\rangle$, then $R^\b{B}(a, b) = 0$. The class $\c{C}(2)$  is the class of $T(2)$-labeled finite triangle free graphs $(\b{B}, \phi)$ so that all of the following hold.
	\begin{enumerate}
		\item 
		No $b\in \b{B}$ has $\phi(b) = \langle 11\rangle$.
		\item
		Whenever $a\neq b\in \b{B}$ and $\phi(a) = \phi(b)$, then either $\phi(a) = \phi(b) = \langle 00\rangle$ or $R^\b{B}(a, b) = 0$.
	\end{enumerate}
	Let $\b{A} = \b{K}_1$ and $f\in \emb(\ct^\b{A}, \ct)$ be as in Example~\ref{Exa:EmbeddingDoesntExtend}. Then $f\not\in \aemb(\ct^\b{A}, \ct)$; if $(\b{C}, \phi)\in \c{C}(0)$ is such that $\b{C}$ has any edges, then $(\b{C}, f\circ \phi)\not\in \c{C}(1)$.
\end{exa}
\vspace{2 mm}

The remainder of the section is spent proving various properties about age maps and aged embeddings. In particular, we will see that aged embeddings exist and can always be extended.
\vspace{2 mm}

\begin{defin}
	\label{Def:PrimeMap}
	Suppose $S\subseteq T$, and let $f\colon S\to T$ be any map. Then $f'\colon \is(S)\to T$ is the map given by $f'(s^\frown i) = f(s)^\frown i$.
\end{defin}

\begin{rem}
	While very simple, the above definition is one of the main ways in which we use the assumption that $\c{L}$ is a binary language. If $\c{L}$ were a higher arity language, then we would need to use a coding tree with more and more rapid branching rather than constant branching. In turn, when attempting a na\"ive generalization of these proof methods to higher arities, one has several possible choices of map analogous to $f'$ defined above. Unfortunately, these different choices can be used to create counterexample colorings, showing that the na\"ive generalization of Theorem 3.1 does not hold in general.
\end{rem}
\vspace{2 mm}

\begin{prop}
	\label{Prop:PrimeAgeMap}
	Suppose $\b{A}$ is an enumerated structure with $\age{\b{A}}\subseteq \c{K}$, and let $f\in \aemb(\ct^\b{A}|_m, \ct)$ for some $m\leq  |\b{A}|$, $m< \omega$. Then $f'\colon (T(m), \c{C}^\b{A})\to T(\tilde{f}(m-1)+1)$ is an age map.
\end{prop}

\begin{proof}
	In one direction, suppose $(\b{B}, \phi)$ is a $T(m)$-labeled $\c{L}$-structure with $\b{B}[f'\circ \phi]\in \c{K}$. To see that $\b{B}[\phi, \b{A}]\in \c{K}$, we observe that it is isomorphic to the induced substructure of $\b{B}[f'\circ \phi]$ on the subset $\b{B}\cup \{\tilde{f}(0),...,\tilde{f}(m-1)\}$.
	
	For the more difficult direction, suppose $(\b{B}, \phi)\in \c{C}^\b{A}(m)$, i.e.\ that $\b{B}[\phi, \b{A}]\in \c{K}$. Let $x\not\in \omega\cup \b{B}$ be some new vertex, and let $\b{C}$ be the structure on $\b{B}\cup \{x\}$ which is isomorphic to $\b{B}\cup \{m-1\}$ viewed as an induced substructure of $\b{B}[\phi, \b{A}]$. Let $\psi\colon \b{C}\to T(m-1)$ be the labeling with $\psi(b) = \pi_{m-1}\circ \phi(b)$ for $b\in \b{B}$, and $\psi(x) = c^\b{A}(m-1)$. Then $\b{C}[\psi, \b{A}] \cong \b{B}[\phi, \b{A}]$, so is in $\c{K}$. Hence $\b{C}[f\circ \psi]\in \c{K}$, and we show that this structure is isomorphic to $\b{B}[f'\circ \phi]$. 
	
	To see this, we first note that these structures are identical on their common subset $\b{B}\cup \b{K}_{\tilde{f}(m-1)}$. Second, we have $f(c^\b{A}(m-1)) = c(\tilde{f}(m-1))$ and $U^\b{C}(x) = U(m-1)$, implying that as an induced substructure of $\b{C}[f\circ \psi]$, we have $\b{K}_{\tilde{f}(m-1)}\cup \{x\}\cong \b{K}_{\tilde{f}(m-1)+1}$. Second, we note that for each $b\in \b{B}$, we have
	\begin{align*}
	R^{\b{C}[f\circ \psi]}(b, x) &= R^\b{C}(b, x) \\
	&= R^{\b{B}[\phi, \b{A}]}(b, m-1) \\
	&= R^{\b{B}[f'\circ \phi]}(b, \tilde{f}(m-1)).
	\end{align*} 
	Hence identifying the points $x\in \b{C}[f\circ \psi]$ and $\tilde{f}(m-1)\in \b{B}[f'\circ \phi]$ yields the desired isomorphism.
\end{proof}
\vspace{2 mm}

The next proposition, though easy, is of vital importance. It gives us a safe way in which to extend partial constructions; this safe extension is to move up and to the left. In particular, this proposition crucially uses the fact that $\c{K}$ is a free amalgamation class. If $\c{K}$ were just a strong amalgamation class, the potential absence of a universally safe direction to extend age maps can break a key part of the proof of Theorem~\ref{Thm:HL}, namely Lemma~\ref{Lem:LiftMagicConditionSpl}.  
\vspace{2 mm}

\begin{prop}
	\label{Prop:ExtendAgeMap}
	Suppose $S\subseteq T(m)$ (possibly empty), $n \geq m$, and $\gamma\colon S\to T(n)$ is an age map so that $\gamma(s)\sqsupseteq s$ for each $s\in S$. Extend the domain of $\gamma$ to all of $T(m)$ by setting $\gamma(t) = \Left(t, n)$ for $t\in T(m)\setminus S$. Then $\gamma\colon T(m)\to T(n)$ is an age map.
\end{prop}

\begin{proof}
	In one direction, suppose $(\b{B}, \phi)$ is a $T(m)$-labeled $\c{L}$-structure with $\b{B}[\gamma\circ \phi]\in \c{K}$. Then $\b{B}[\phi]$ is an induced substructure of $\b{B}[\gamma\circ \phi]$, so is also in $\c{K}$. 
	
	In the other direction, suppose $(\b{B}, \phi)\in \c{C}(m)$. Viewing $\phi^{-1}(S)$ as an induced substructure of $\b{B}$, we have $(\phi^{-1}(S))[\gamma\circ \phi] \in \c{K}$. Then we note as before that $\b{B}[\phi]$ is an induced substructure of $\b{B}[\gamma\circ \phi]$. Finally, we note that in $\b{B}[\gamma\circ \phi]$, there are no relations between $\phi^{-1}(T(m)\setminus S)$ and $\{m,...,n-1\}$. Hence $\b{B}[\gamma\circ \phi]$ is a free amalgam of structures in $\c{K}$, so also in $\c{K}$.
\end{proof}
\vspace{2 mm}

\begin{theorem}
	\label{Thm:ExistsAgedEmbs}
	Suppose $\b{A}$ is an enumerated structure with $\age{\b{A}}\subseteq \c{K}$, and fix $m< |\b{A}|$. If $f\in \aemb(\ct^\b{A}|_m, \ct)$ and $\gamma\colon f'[T(m)]\to T(n)$ is an age map with $\gamma(f'(s))\sqsupseteq f'(s)$ for each $s\in T(m)$, then there is $g\in \aemb(\ct^\b{A}|_{m+1}, \ct)$ with $g|_{T(<m)} = f$ and $g(t)\sqsupseteq \gamma(f'(t))$ for each $t\in T(m)$. 
	
	In particular, there is an aged embedding of $\ct^\b{A}$ into $\ct$.
\end{theorem}

\begin{proof}
	Suppose $s = c^\b{A}(m)$. Find $r\geq n$ so that $\Left(\gamma(f'(s)), r)$ is a coding node of type $u^\b{A}(m)$; this is possible by Proposition~\ref{Prop:PrimeAgeMap} and our assumption that $\b{K}$ is left dense. We then define $g$ on $T(m)$ by setting $g(t) = \Left(\gamma(f'(t)), r)$ for each $t\in T(m)$. By Propositions \ref{Prop:PrimeAgeMap} and \ref{Prop:ExtendAgeMap}, $g\colon T(m)\to T(r)$ is an age map, hence $g$ is an aged embedding as desired.
	
	For the last statement, we use left density to find some $\ell< \omega$ so that $0^\ell = c(\ell)$ and $u(\ell) = u^\b{A}(0)$. Then by setting $f(\emptyset) = 0^\ell$, we obtain a member of $\aemb(\ct^\b{A}|_1, \ct)$, allowing us to start the induction.
\end{proof}

\section{Ramsey theorems for aged embeddings}
\label{Sec:RamseyAE}

This section proves analogs of the Halpern--L\"auchli and Milliken theorems (Theorems \ref{Thm:HL} and \ref{Thm:Milliken}, respectively) for coding trees with aged embeddings. The proof of Theorem~\ref{Thm:HL} takes up most of the section, and uses ideas from forcing. The reader might want to take a look at the appendix on forcing for the conventions and notation used in the proof. In particular, the notion of ``name for an ultrafilter" that we use (Definition~\ref{Def:ForcingFilterNames}) is more restrictive, but better suited for our purposes and for readers less familiar with forcing.
\vspace{2 mm}
	
\begin{theorem}
	\label{Thm:HL}
	Suppose $\b{A}$ is an enumerated structure with $\age{\b{A}}\subseteq \c{K}$, and fix $m< |\b{A}|$. Suppose $f\in \aemb(\ct^\b{A}|_m, \ct)$, and let 
	$$F := \{g\in \aemb(\ct^\b{A}|_{m+1}, \ct): g|_{T(<m)} = f\}.$$
	Let $\chi\colon F\to \im{\chi}$ be a finite coloring. Then there is $h\in \aemb(\ct, \ct)$ with $h|_{T(\leq\tilde{f}(m-1))} = \mathrm{id}$ and $h\circ F$ monochromatic for $\chi$.
\end{theorem}

\begin{rem}
	The proof also works when $m = 0$, with $f = \emptyset$ and $F = \aemb(\ct^\b{A}|_1, \ct)$. We treat the demand that $h|_{T(\leq \tilde{f}(m-1))} = \mathrm{id}$ as vacuous, and we set $f'(\emptyset) = \emptyset$.
\end{rem}

\begin{rem}
	If one removes the coding tree structure and replaces ``aged embedding" with ``strong similarity," one recovers a version of the Halpern-L\"auchli theorem.
\end{rem}

\begin{proof}
	The overall strategy of the proof of Theorem~\ref{Thm:HL} is as follows. First, we define a large forcing poset; intuitively, this poset views $F$ as a tree-like structure and describes branches through this tree. Certain levels of these branches correspond to members of $F$, so to such a branch and level, we can associate the color that $\chi$ gives this member of $F$. In this way, we can partition the levels of a branch into finitely many pieces, and a non-principal ultrafilter on $\omega$ can decide which piece is largest. The reason we work with such a large poset is to allow us to use the Erd\H{o}s-Rado theorem; this allows us to find a neighborhood of branches (i.e.\ those branches passing through a particular member of $F$) where this level partition has similar behavior. Working in this neighborhood, we are able to construct $h$ as desired.

	Set $N = k^m$, writing $T(m) = \{s_i: i < N\}$, and suppose that $c^\b{A}(m) = s_d$. Let $\kappa$ be a suitably large infinite cardinal;  $\kappa = (\beth_{2N-1})^+$ suffices. We define a poset $\langle \mathbb{P}, \leq_\bb{P}\rangle$ as follows.
	
	\begin{itemize}
		\item 
		Elements of $\mathbb{P}$ are functions $p\colon B(p)\times N\to T(\ell(p))$ satisfying the following:
		\begin{enumerate}
			\item 
			$B(p)\subseteq \kappa$ is finite.
			\item 
			$\ell(p) \geq \tilde{f}(m-1)+1$ satisfies $c(\ell(p))\sqsupseteq f'(s_d)$ and $u(\ell(p)) = u^\b{A}(m)$.
			\item 
			$p(\alpha, i)\sqsupseteq f'(s_i)$ for each $\alpha\in B(p)$ and $i< N$.
			\item 
			For each $\alpha\in B(p)$, $p(\alpha, d) = c(\ell(p))$.
		\end{enumerate} 
		\item
		Given $p, q\in \mathbb{P}$, we declare that $q\sqsupseteq p$ if the following hold:
		\begin{enumerate}
			\item 
			$B(p)\subseteq B(q)$ and $\ell(p)\leq \ell(q)$.
			\item 
			For $\alpha\in B(p)$ and $i< N$, we have $q(\alpha, i)\sqsupseteq p(\alpha, i)$.
		\end{enumerate}
		\item 
		Suppose $q\sqsupseteq p$. Let $S\subseteq B(p)\times N$ be any subset on which $p$ is injective. Write $\theta(p, q, S)\colon p[S]\to q[S]$ for the map defined by $\theta(p, q, S)(p(\alpha, i)) = q(\alpha, i)$, where $(\alpha, i)\in S$. 
		\item 
		We declare that $q\leq_\bb{P} p$ (that $q$ is \emph{stronger than} $p$ or \emph{extends} $p$) iff $q\sqsupseteq p$ and for any $S\subseteq B(p)\times N$ on which $p$ is injective, we have that $\theta(p, q, S)$ is an age map. 
	\end{itemize}

One key point in the definition of $\langle \bb{P}, \leq_\bb{P}\rangle$ is the following: members of $F$ are determined by their value on $T(m)$. So if $p\in \mathbb{P}$ and $\{\alpha_i : i < N\}\subseteq B(p)$ are such that $\{p(\alpha_i, i): i< d\}\in F$ and $q\leq_\bb{P} p$, then also $\{q(\alpha_i, i): i< d\}\in F$. In particular, we will be able to talk about the color that $\chi$ assigns to such tuples. This ``diagonalization'' will occur quite frequently, so given $\vec{\alpha} = \{\alpha_i: i < N\}\subseteq \kappa$, we write $p(\vec{\alpha}) := \{p(\alpha_i, i): i< N\}$, and if $n\leq \ell(p)$, we write $p(\vec{\alpha})|_n := \{p(\alpha_i, i)|_n: i< N\}$. In what follows, $\vec{\alpha} = \{\alpha_i: i< N\}$ always denotes a subset of $\kappa$ of size $N$ with $\alpha_0 < \cdots < \alpha_{N-1}$.

For each $\vec{\alpha}$ and $\epsilon\in \im{\chi}$, we set
\begin{align*}
\dot{b}(\vec{\alpha}) &:= \{(q, \ell(q)): \vec{\alpha}\subseteq B(q),\, q(\vec{\alpha})\in F\},\\
\dot{b}(\vec{\alpha}, \epsilon) &:= \{(q, \ell(q)): \vec{\alpha}\subseteq B(q),\, q(\vec{\alpha})\in F,\, \chi(q(\vec{\alpha})) = \epsilon\}.
\end{align*}
So $\dot{b}(\vec{\alpha})$ and $\dot{b}(\vec{\alpha}, \epsilon)$ are names for subsets of $\omega$. For each $p\in \mathbb{P}$, we define
$$\dot{L}(p) := \{(q, \ell(q)): q\leq_\bb{P} p\}.$$
So $\dot{L}(p)$ names a subset of $\omega$, and $p\Vdash ``\dot{L}(p) \text{ is infinite}"$. We then set 
$$\dot{\c{G}} := \{(p, \dot{L}(p)): p\in \mathbb{P}\}.$$
So $\dot{\mathcal{G}}$ names a collection of infinite subsets of $\omega$. It turns out that $\mathbb{P}\Vdash ``\dot{\mathcal{G}} \text{ has the SFIP}"$ (see Definition~\ref{Def:ForcingFilterNames}). To see this, if we are given $p_0,...,p_{n-1}\in \bb{P}$ and $q\in \bb{P}$ with $q\leq_\bb{P} p_i$ for each $i< n$, then $q \Vdash ``\dot{L}(q)\subseteq \dot{L}(p_i)"$ for each $i< n$.

Let $\dot{\mathcal{U}}$ be a name for some non-principal ultrafilter extending $\dot{\mathcal{G}}$.
\vspace{2 mm}

\begin{lemma}
	\label{Lem:FindConditionsSpl}
	For each $\vec{\alpha}$, there are $q_{\vec{\alpha}}\in \mathbb{P}$ and $\epsilon_{\vec{\alpha}}\in \im{\chi}$ so that:
	\begin{enumerate}
		\item 
		$\vec{\alpha}\subseteq B(q_{\vec{\alpha}})$,
		\item 
		$q_{\vec{\alpha}}\Vdash ``\dot{b}(\vec{\alpha}, \epsilon_{\vec{\alpha}})\in \dot{\mathcal{U}}".$
	\end{enumerate}
\end{lemma}

\begin{proof}
	We first define a condition $p_{\vec{\alpha}}$ by first fixing some $g\in F$, the same $g$ for all $\vec{\alpha}$. We set $B(p) = \vec{\alpha}$, and for $i, j < N$, we set $p_{\vec{\alpha}}(\alpha_i, j) = g(s_j)$.
	Notice in particular that $p_{\vec{\alpha}}(\vec{\alpha})\in F$. This implies that $p_{\vec{\alpha}}\Vdash ``\dot{L}(p_{\vec{\alpha}})\subseteq \dot{b}(\vec{\alpha}),"$ so in particular $p_{\vec{\alpha}}\Vdash ``\dot{b}(\vec{\alpha})\in \dot{\c{U}}."$ Then find $q_{\vec{\alpha}}\leq_\bb{P} p_{\vec{\alpha}}$ and $\epsilon_{\vec{\alpha}}\in \im{\chi}$ with $q_{\vec{\alpha}}\Vdash ``\dot{b}(\vec{\alpha}, \epsilon_{\vec{\alpha}})\in \dot{\c{U}}."$
\end{proof}
\vspace{2 mm}

We are now prepared to apply the Erd\H{o}s-Rado theorem.
\vspace{2 mm}

\begin{lemma}
	\label{Lem:ErdosRado}
	Let $q_{\vec{\alpha}}$ and $\epsilon_{\vec{\alpha}}$ be as in Lemma~\ref{Lem:FindConditionsSpl}. There are countably infinite subsets $K_0 < \cdots < K_{N-1}$ of $\kappa$ so that the following hold.

\begin{enumerate}
	\item 
	There are $\ell^*< \omega$ and $\epsilon^* \in \im{\chi}$ so that for every $\vec{\alpha}\in \prod_{i< N} K_i$, we have $\ell(q_{\vec{\alpha}}) = \ell^*$ and $\epsilon_{\vec{\alpha}} = \epsilon^*$. 
	
	\item 
	With $\ell^*$ as in item 1, there are $t_i\in T(\ell^*)$ so that for every $\vec{\alpha}\in \prod_{i< N} K_i$, we have $q_{\vec{\alpha}}(\alpha_i, i) = t_i$. 
	
	\item 
	If $J\subseteq \prod_{i< N} K_i$ is finite, then $\bigcup_{\vec{\alpha}\in J} q_{\vec{\alpha}}\in \mathbb{P}$.
\end{enumerate}
\end{lemma}

\begin{proof}
	Recall that by the Erd\H{o}s-Rado theorem, we have $\kappa\to (\aleph_1)^{2N}_{\aleph_0}$. This means that for any coloring of $[\kappa]^{2N}$ in $\aleph_0$-many colors, there is a subset $X\subseteq \kappa$ of size $\aleph_1$ on which the coloring is monochromatic.
	
	If $\vec{\alpha}\in [\kappa]^N$, we let $i_{\vec{\alpha}}\colon |B(q_{\vec{\alpha}})|\to B(q_{\vec{\alpha}})$ denote the increasing bijection, and we let $q_{\vec{\alpha}}\cdot i_{\vec{\alpha}}\colon |B(q_{\vec{\alpha}})|\times N\to T(\ell(q_{\vec{\alpha}}))$ be the map given by $q_{\vec{\alpha}}\cdot i_{\vec{\alpha}}(n, i) = q_{\vec{\alpha}}(i_{\vec{\alpha}}(n), i)$. If $S\in [\kappa]^{2N}$, we let $i_S\colon |\bigcup_{\vec{\alpha}\in [S]^N} B(q_{\vec{\alpha}})|\to \bigcup_{\vec{\alpha}\in [S]^N} B(q_{\vec{\alpha}})$ denote the increasing bijection. We also let $j_S\colon 2N\to S$ denote the increasing bijection.
	
	Given $S\in [\kappa]^{2N}$, we define 
	$$\Theta(S) = \{(A,\, \epsilon_{j_S[A]},\, q_{j_S[A]}\cdot i_{j_S[A]},\, i_S^{-1}(B(q_{j_S[A]})),\, i_S^{-1}(j_S[A])): A\in [2N]^N\}.$$
	Then $\Theta$ is a map on $[\kappa]^{2N}$ with countable image, and we may use the Erd\H{o}s-Rado theorem to find $X\subseteq \kappa$ of size $\aleph_1$ on which $\Theta$ is monochromatic. Now let $K_0 < \cdots < K_{N-1}$ be subsets of $X$, each in order type $\omega$, so that for any $i< N$ and between any consecutive members of $K_i$, we can find another member of $X$.
	
	It is immediate that item $1$ holds. For item $2$, we note that for each fixed $A\in [2N]^N$, we have that the set $i_{j_S[A]}^{-1}(j_S[A])$ is the same for every $S\in [X]^{2N}$. So for any $\vec{\alpha}\in [X]^N$, let $S\in [X]^{2N}$ be chosen so that $\vec{\alpha} = j_S[N]$, where $N = \{0,...,N-1\}\in [2N]^N$. For each $i< N$, let $n_i< \omega$ be such that $i_{j_S[N]}(n_i) = \alpha_i$. Since $n_i$ does not depend on $S$, we have $t_i = q_{\vec{\alpha}}(\alpha_i, i) = q_{j_S[N]}\cdot i_{j_S[N]}(n_i, i)$ the same for every $\vec{\alpha}\in [X]^N$.
	
	For item 3, it suffices to show that if $\vec{\alpha}, \vec{\beta}\in \prod_{i< N} K_i$ and $\delta\in B(q_{\vec{\alpha}})\cap B(q_{\vec{\beta}})$, then for each $i< N$ we have $q_{\vec{\alpha}}(\delta, i) = q_{\vec{\beta}}(\delta, i)$. To show this, let $a< \omega$ be such that $i_{\vec{\alpha}}(a) = \delta$. Then we want $i_{\vec{\beta}}(a) = \delta$ as well. We know that $i_{\vec{\beta}}(b) = \delta$ for some $b< \omega$. Let $\vec{\gamma}\in [X]^N$ be chosen so that $\gamma_i$ is strictly between $\alpha_i$ and $\beta_i$ unless $\alpha_i = \beta_i$, in which case $\gamma_i = \alpha_i = \beta_i$. Let $Q, R, S\in [X]^{2N}$ be chosen with $\vec{\alpha}\cup \vec{\beta}$, $\vec{\alpha}\cup \vec{\gamma}$ and $\vec{\gamma}\cup \vec{\beta}$, respectively, as initial segments. Let $A, B\in [2N]^N$ be the sets with $j_Q[A] = \vec{\alpha}$, $j_Q[B] = \vec{\beta}$, $j_R[A] = \vec{\alpha}$, $j_R[B] = \vec{\gamma}$, $j_S[A] = \vec{\gamma}$, and $j_S[B] = \vec{\beta}$. Since $i_{j_Q[A]}(a) = i_{j_Q[B]}(b)$, this is also true for $R$ and $S$. But then $i_{j_Q[A]}(a) = i_{j_R[A]}(a) = i_{j_R[B]}(b)$ and $i_{j_Q[B]}(b) = i_{j_S[B]}(b) = i_{j_S[A]}(a)$. But since $j_R[B] = j_S[A] = \vec{\gamma}$, we must have $a = b$ as desired.
\end{proof}
\vspace{2 mm}

We now turn towards the construction of $h\in \aemb(\ct, \ct)$ from the statement of Theorem~\ref{Thm:HL}, which proceeds inductively level by level. To assist us, we will simultaneously build aged embeddings 
$$\psi_n\in \aemb(\ct|_{n+1}, \ct)$$
for every $n \geq \tilde{f}(m-1)+1$. The choice of indexing is because the domain of $\psi_n$ is $T({\leq}n)$.

Letting 
$$\eta\colon (T(m), \c{C}^\b{A})\to T(\ell^*)$$
denote the map with $\eta(s_i) = t_i$, we have that $\eta$
is an age map. Since 
$$f'\colon (T(m), \c{C}^\b{A})\to T(\tilde{f}(m-1)+1)$$
is an age map by Proposition~\ref{Prop:PrimeAgeMap}, we obtain an age map
$$\gamma\colon f'[T(m)]\to T(\ell^*)$$
by setting $\gamma(f'(s_i)) = t_i$. Note that $\gamma(x)\sqsupseteq x$ for each $x\in f'[T(m)]$. We use Proposition~\ref{Prop:ExtendAgeMap} to extend the domain of $\gamma$ to all of $T(\tilde{f}(m-1)+1)$. We now use Theorem~\ref{Thm:ExistsAgedEmbs} to find an aged embedding 
$$\psi_{\tilde{f}(m-1)+1}\in \aemb(\ct|_{\tilde{f}(m-1)+2}, \ct)$$
which is the identity on $T(\leq \tilde{f}(m-1))$ and with $\psi_{\tilde{f}(m-1)+1}(x)\sqsupseteq \gamma(x)$ for each \newline $x\in T(\tilde{f}(m-1)+1)$.

Set $F(m):= \{\tilde{g}(m): g\in F\}$. Suppose the aged embedding $\psi_n$ has been defined for some $n\geq \tilde{f}(m-1)+1$ so that $\psi_n(x)\sqsupseteq \gamma(x)$ for each $x\in f'[T(m)]$. Two cases emerge. In the easy case, where $n\not\in F(m)$, we set $h|_{T(\leq n)} = \psi_n$, and let $\psi_{n+1}$ be any aged embedding extending $\psi_n$.

The difficult case is when $n\in F(m)$. Let $F_n = \{g\in F: g(m) = n\}$, and set $S_i := \{\psi_n\circ g(s_i): g\in F_n\}\subseteq T(\tilde{\psi}_n(n))$. Let $\rho_i\colon S_i\to K_i$ be any injection. For each $g\in F_n$, let $$\vec{\alpha}_g := \{\rho_i\circ \psi_n\circ g(s_i): i< k^m\},$$
and set $q_g := q_{\vec{\alpha}_g}$. Set $q = \bigcup_{g\in F_n} q_g$. So by Lemma~\ref{Lem:ErdosRado}, we have $q\in \mathbb{P}$. We now define a condition $r\in \mathbb{P}$ as follows.
\begin{enumerate}
	\item 
	$B(r) = B(q) = \bigcup_{g\in F_n} B(q_g)$, and $\ell(r) = \tilde{\psi}_n(n)$.
	\item 
	If $\alpha\in B(r)$ and there is $g\in F_n$ with $\alpha = \rho_i\circ \psi_n\circ g(s_i)$, then set $r(\alpha, i) = \psi_n\circ g(s_i)$.
	\item 
	Otherwise, set $r(\alpha, i) = \Left(q(\alpha, i), \tilde{\psi}_n(n))$ for $i\neq d$, and set $r(\alpha, d) = \psi_n\circ g(s_d)$ for any $g\in F_n$ (well defined since $g(s_d) = c(n)$ for every $g\in F_n$).
\end{enumerate}
\vspace{2 mm}

\begin{lemma}
	\label{Lem:LiftMagicConditionSpl}
	The condition $r$ is well defined, and $r\leq_\bb{P} q_g$ for each $g\in F_n$.
\end{lemma}

\begin{proof}
	To see that $r$ is well defined, we note that if $\alpha = \rho_i\circ \psi_n\circ g(s_i)$ for more than one $g\in F_n$, then these $g$ must agree on $s_i$, as $\rho_i$ and $\psi_n$ are injective. We note that $r(\alpha, i)\sqsupseteq f'(s_i)$ for each $\alpha\in B(r)$ and $i< N$. We also have that $r(\alpha, d) = c(\tilde{\psi}_n(n))$ and $u(\tilde{\psi}_n(n)) = u^\b{A}(m)$ since $\psi_n\in \aemb(\ct|_{n+1}, \ct)$. Hence $r\in \mathbb{P}$. 
	
	Fix $g\in F_n$. To see that $r\leq_\bb{P} q_g$, we note first that $r\sqsupseteq q_g$. Let $S\subseteq B(q_g)\times N$ be a subset on which $q_g$ is injective, and form the map $\theta := \theta(q_g, r, S)\colon q_g[S]\to r[S]$. For a subset $V\subseteq q_g[S]\cap \{t_i: i< N\}$ with $t_d\in V$, we have $\theta(t_i) = g(s_i)$. So $\theta\colon V\to r[S]$ is an age map  since $g\colon (T(m), \c{C}^\b{A})\to T(\ell^*)$ and $\eta$ are age maps. Outside of $V$, $\theta$ is just the left successor map, and we use Proposition~\ref{Prop:ExtendAgeMap} to conclude that $\theta$ is an age map.
\end{proof}
\vspace{2 mm}

Since $r\leq_\bb{P} q_g$ for each $g\in F_n$, we have $r\Vdash ``\dot{b}(\vec{\alpha}_g, \epsilon^*) \in \dot{\c{U}}."$ Since we also have $r\Vdash ``\dot{L}(r)\in \dot{\c{U}},"$ we may find $y_0\leq_\bb{P} r$ and some fixed $M> \tilde{\psi}_n(n)$ so that $y_0\Vdash ``M\in \dot{L}(r)"$ and $y_0\Vdash ``M\in \dot{b}(\vec{\alpha}_g, \epsilon^*)"$ for each $g\in F_n$.
Then strengthen to  $y_1\leq_\bb{P} y_0$ in order to find $y\geq_\bb{P} y_1$ and $y_g\geq_\bb{P} y_1$ with $(y, M)\in \dot{L}(r)$ and $(y_g, M)\in \dot{b}(\vec{\alpha}, \epsilon^*)$. In particular, we have $\ell(y) = \ell(y_g) = M$, $y\leq_\bb{P} r$, and $\chi(y_g(\vec{\alpha}_g)) = \epsilon^*$. But since $y_1$ strengthens both $y$ and $y_g$, we must also have $y(\vec{\alpha}_g) = y_g(\vec{\alpha}_g)$. So also $\chi(y(\vec{\alpha}_g)) = \epsilon^*$.

For ease of notation set $\alpha_g(i) = \rho_i\circ \psi_n\circ g(s_i)$. So by the definition of $r$, we have $r(\alpha_g(i), i) = \psi_n\circ g(s_i)$. Notice that if $r(\alpha_{g_0}(i), i) = r(\alpha_{g_1}(j), j)$ for $g_0, g_1\in F_n$ and $i, j< d$, then $i = j$ and $\alpha_{g_0}(i) = \alpha_{g_1}(i)$. Hence the map $$\xi\colon \{r(\alpha_g(i), i): g\in F_n, i< d\}\to T(M)$$
given by $\xi(r(\alpha_g(i), i)) = y(\alpha_g(i), i)$ is well defined. Since $y\leq_\bb{P} r$, $\xi$ is an age map. Extend the domain of $\xi$ to all of $\psi_n[T(n)]$ using Proposition~\ref{Prop:ExtendAgeMap}. Noting that for any $g\in F_n$, we have $\xi(c(\tilde{\psi}_n(n))) = \xi(r(\alpha_g(d), d)) = y(\alpha_g(d), d) = c(M)$, it follows that we can define $h|_{T(\leq n)}$ on $T(n)$ by setting $h(t) = \xi(\psi_n(t))$. We then let $\psi_{n+1}\in \aemb(\ct|_{n+2}, \ct)$ be any extension of $h|_{T(\leq n)}$.

This concludes the construction of $h$ and the proof of Theorem~\ref{Thm:HL}.
\end{proof}
\vspace{2 mm}

\begin{theorem}
	\label{Thm:Milliken}
	Suppose $\b{A}\in \c{K}$ is an enumerated structure. Let $\chi\colon \aemb(\ct^\b{A}, \ct)\to 2$ be a coloring. Then there is $h\in \aemb(\ct, \ct)$ with $h\cdot \aemb(\ct^\b{A}, \ct)$ monochromatic.
\end{theorem}

\begin{rem}
	If one removes the coding tree structure and replaces ``aged embedding" with ``strong similarity," one recovers a version of Milliken's theorem.
\end{rem}

\begin{proof}
	We induct on $|\b{A}|$. When $|\b{A}| = 1$, this follows directly from Theorem~\ref{Thm:HL}.
	
	Assume the theorem is true if $|\b{A}|\leq m$, and suppose $|\b{A}| = m+1$. Enumerate $\aemb(\ct^\b{A}|_m, \ct)$ as $\{f_i: i< \omega\}$ in such a way that whenever $i< j < \omega$, we have $\tilde{f}_i(m-1)\leq \tilde{f}_j(m-1)$. For ease of notation, set $n_i := \tilde{f}_i(m-1)$. Let \newline $F_i = \{g\in \aemb(\ct^\b{A}, \ct): g \text{ extends } f_i\}$.
	
	We inductively define for each $i< \omega$ an aged embedding $h_i\in \aemb(\ct, \ct)$ as follows. Use Theorem~\ref{Thm:HL} on $f_0$ and $F_0$ to obtain $h_0$ with $h_0|_{T(\leq n_0)} = \mathrm{id}$ and $h_0\cdot F_0$ monochromatic for $\chi$, say with color $j_0\in 2$. If $h_0,..., h_{i-1}$ have been defined, then write $\psi_i = h_0\circ \cdots \circ h_{i-1}$. Then use Theorem~\ref{Thm:HL} to obtain $h_i$ with $h_i|_{T(\leq n_i)} = \mathrm{id}$ and $h_i\cdot F_i$ monochromatic for $\chi\cdot \psi_i$, say with color $j_i\in 2$. 
	
	We notice that for any $t\in T$ and any sufficiently large $i< \omega$, $h_i(t) = t$. It follows that the sequence $(\psi_i(t))_i$ stabilizes, and we define $\phi\in \aemb(\ct, \ct)$ by setting $\phi(t) = \lim_i \psi_i(t)$. Then $\phi$ is an aged embedding since each $\psi_i$ is an aged embedding. Suppose $g\in F_i$. Then for any $n\geq i$, we also have $(h_i\circ\cdots \circ h_n)\cdot g\in F_i$. It follows that by taking $n$ large enough, we have $\chi\cdot\phi(g) = \chi\cdot \psi_i( (h_i\circ \cdots \circ h_n)\cdot g) = j_i$. 
	
	Hence we can now view $\chi\cdot \phi$ as a coloring of $\aemb(\ct^\b{A}|_m, \ct)$ by setting $\chi\cdot \phi(f_i) = j_i$. Using our inductive hypothesis, find $\xi\in \aemb(\ct, \ct)$ with $\xi\cdot \aemb(\ct^\b{A}|_m, \ct)$ monochromatic for $\chi\cdot \phi$. We now set $h = \phi\cdot \xi$.
\end{proof}

\section{Envelopes}
\label{Sec:Envelopes}

In order to apply Theorem~\ref{Thm:Milliken} and obtain upper bounds for big Ramsey degrees, we need to understand which subsets of $\omega$ are of the form $\im{\tilde{f}}$ for some $f\in \aemb(\ct^\b{A}, \ct)$ and $\b{A}\in \c{K}$. We will call sets of this form \emph{envelopes}. This leads to the notion of the \emph{closure} $\overline{S}$ of a finite subset $S\subseteq \omega$; this is just the smallest superset of $S$ which is an envelope. It will be crucial to understand how large $\overline{S}$ can be compared to $S$. In particular, we develop an abstract criterion in terms of the sizes of closures of finite $S\subseteq \omega$, Theorem~\ref{Thm:EnvelopesImplyDegrees}, which implies that $\c{K}$ has finite big Ramsey degrees.
\vspace{2 mm}

\begin{defin}
	\label{Def:Envelope}
	Suppose $S\subseteq \omega$. Letting $i_S\colon |S|\to S$ be the increasing bijection, we set $\b{K}_S := \b{K}\cdot i_S$, $\ct^S := (c^S, u^S) := \ct^{\b{K}_S}$, and $\c{C}^S = \c{C}^{\b{K}_S}$.
	
	If $S$ is finite, we say that $S$ is an \emph{envelope} if there is $f\in \aemb(\ct^S, \ct)$ with $\tilde{f} = i_S$.
\end{defin}
\vspace{2 mm}

\begin{prop}
	\label{Prop:ComboEnvelope}
	Let $S\subseteq \omega$ be finite. Then $S$ is an envelope iff $S$ satisfies both of the following conditions.
	\begin{enumerate}
		\item 
		For any $m, n\in S$, we have $\ell(c(m)\wedge c(n))\in S$.
		\item 
		For each $m \leq \max(S)$, set $c[S]|_m = \{c(a)|_m: a\in S\setminus  m\}$. Then if $m\not\in S$, we have that $\pi_m\colon c[S]|_{m+1}\to T(m)$ is an age map.
	\end{enumerate} 
\end{prop}

\begin{rem}
	Note that item 2 can be rephrased as follows: let $n_0 < n_1$ be consecutive elements of $S$. Then $\pi_{n_0} : c[S]|_{n_1}\to T(n_0+1)$ is an age map.
\end{rem}

\begin{proof}
	First assume $S$ is an envelope, as witnessed by $f\in \aemb(\ct^S, \ct)$. Item $1$ is clear. For item $2$, suppose $n_0 < n_1$ are consecutive elements of $S$. Then if $n_1 = i_S(m)$, we have $c[S]|_{n_1} = f[V]$ for some $V\subseteq T(m)$. The map $\pi_{n_0+1}\colon c[S]|_{n_1}\to T(n_0+1)$ then becomes the map $\pi_{n_0+1}\colon f[V]\to f'[V]$, so is injective. It follows from Proposition~\ref{Prop:PrimeAgeMap} that it is an age map. 
	
	Now assume that $S$ satisfies items $1$ and $2$. Define $f\colon T(<\!|S|)\to T$ as follows
	\begin{itemize}
		\item 
		We set $f(\emptyset) = c(i_S(0))$. Note that $\tilde{f}(0) = i_S(0)$.
		\item 
		Assume $f$ has been defined on $T(<m)$ for some $0 < m < |S|$, with $\tilde{f}(j) = i_S(j)$ for each $j< m$. If $t\in T(m-1)$ and $i< k$, we set
		\begin{align*}
		f(t^\frown i) = \begin{cases}
		c(n)|_{i_S(m)} \quad &\text{if $\exists n\in S$ with } c(n)\sqsupseteq f(t)^\frown i,\\
		\Left(f(t)^\frown i, i_S(m)) \quad &\text{otherwise.}
		\end{cases}
		\end{align*}
		Note that $\tilde{f}(m) = i_S(m)$.
	\end{itemize} 
	This is well defined by item 1. The fact that $\tilde{f} = i_S\in \emb(\b{K}_S, \b{K})$ has two useful consequences. First, we have $u^S(m) = u(\tilde{f}(m))$ for each $m < |S|$. Second, one can show by induction on $m< |S|$ that for every $n> m$, $t\in T(m)$, and $i< k$, we have that $c^S(n)\sqsupseteq t^\frown i$ iff $c(i_S(n))\sqsupseteq f(t)^\frown i$.

	We will show that $f\in \aemb(\ct^S, \ct)$ by induction on level. For $m = 0$, we have $f(c^S(0)) = f(\emptyset) = c(i_S(0))$, and both $(\{\emptyset\}, \c{C}^S)$ and $(\{c(i_S(0))\}, \c{C})$ are all of $\c{K}$. Suppose $f|_{T(<m)}\in \aemb(\ct^S|_m, \ct)$ for some $0 < m < |S|$. We first note that if $c^S(m) = t^\frown i$ for some $t\in T(m-1)$ and $i< k$, then $c(i_S(m)) \sqsupseteq f(t)^\frown i$. So by the definition of $f$, we have $f(c^S(m)) = c(i_S(m))$. To see that $f\colon (T(m), \c{C}^S)\to T(i_S(m))$ is an age map, we first use Proposition~\ref{Prop:PrimeAgeMap} to see that $f'\colon (T(m), \c{C}^S)\to T(i_S(m-1)+1)$ is an age map. Then since $S$ satisfies item 2, we must have that 
	$$\pi_{i_S(m-1)+1}\colon \{c(n)|_{i_S(m)}: n\in S, n\geq i_S(m)\}\to T(i_S(m-1)+1)$$
	is an age map. We use Proposition~\ref{Prop:ExtendAgeMap} to conclude that $f\colon (T(m), \c{C}^S)\to T(i_S(m))$ is an age map as desired.
\end{proof}
\vspace{2 mm}

\begin{cor}
	\label{Cor:EnvelopeImages}
	Suppose $S\subseteq \omega$ is finite and $h\in \aemb(\ct, \ct)$. Then $S$ is an envelope iff $\tilde{h}[S]$ is an envelope.
\end{cor}

\begin{proof}
	If $S\subseteq \omega$ is an envelope as witnessed by $f\in \aemb(\ct^S, \ct)$, then $\tilde{h}[S]$ is an envelope as witnessed by $h\cdot f$, where we note that $\ct^{\tilde{h}[S]} = \ct^S$. Conversely, suppose $\tilde{h}[S]$ satisfies items $1$ and $2$ from Proposition~\ref{Prop:ComboEnvelope}. Then $S$ satisfies item 1 since $h$ respects meets. For item 2, let $n_0< n_1$ be consecutive members of $S$. We consider the age maps $h\colon T(n_1)\to T(\tilde{h}(n_1))$ and $h'\colon T(n_0+1)\to T(\tilde{h}(n_0)+1)$ and note that $h' \circ \pi_{n_0+1} = \pi_{\tilde{h}(n_0)+1} \circ h$. It follows that $\pi_{n_0+1}: c[S]_{n_1}\to T(n_0+1)$ must be an age map.
\end{proof}
\vspace{2 mm}

\begin{defin}
	\label{Def:Closure}
	Let $S\subseteq \omega$ be finite. The \emph{closure} of $S$, denoted $\overline{S}$, is the smallest envelope containing $S$, equivalently the intersection of all envelopes of $S$.
	
	Suppose $S\subseteq \omega$ is an envelope. The \emph{interior} of $S$, denoted $\Int(S)$, is the smallest subset of $S$ with $\overline{\Int(S)} = S$. 
\end{defin}
\vspace{2 mm}

As an immediate consequence of Corollary~\ref{Cor:EnvelopeImages}, we see that if $h\in \aemb(\ct, \ct)$ and $S\subseteq \omega$ is finite, then $\overline{\tilde{h}[S]} = \tilde{h}\left[\overline{S}\right]$. If $S$ is an envelope, then $\Int(\tilde{h}[S]) = \tilde{h}[\Int(S)]$.  
\vspace{2 mm}

\begin{rem}
	Given a finite $S\subseteq \omega$, Proposition~\ref{Prop:ComboEnvelope} gives us the following ``top-down" method of computing $\overline{S}$. Start by setting $S_{\max(S)} = \{\max(S)\}$. If $m< \max(S)$ and $S_{m+1}\subseteq \omega$ has been determined, then we set $S_m = S_{m+1}\cup \{m\}$ if any of the following hold:
	\begin{itemize}
		\item 
		$m\in S$.
		\item 
		There are $n_0, n_1\in S_{m+1}$ with $m = \ell(c(n_0)\wedge c(n_1))$.
		\item 
		The map $\pi_m\colon c[S_{m+1}]|_{m+1}\to T(m)$ is not an age map.
	\end{itemize} 
	If none of the above hold, we set $S_m = S_{m+1}$. Then $S_0 = \overline{S}$.
	
	Similarly, if $S\subseteq \omega$ is an envelope, then we have the following ``top-down" method of computing $\Int(S)$. Start by setting set $\Int(S)_0 = \{\max(S)\}$, and if $\Int(S)_n$ has been determined, set $\Int(S)_{n+1} = \Int(S)_n\cup \{\max\left(S\setminus \overline{\Int(S)_n}\right)\}$. Then $\Int(S)_n$ eventually stabilizes, and $\Int(S) = \Int(S)_n$ for any large enough $n$.
\end{rem}
\vspace{2 mm}

\begin{exa}
\label{Exa:UnboundedEnvelopes}	
	While $\overline{S}$ exists and is finite, we note that there is not a uniform bound on the size of the closure. Let $\c{K}$ be the class of finite triangle-free graphs. Set $n_0 = 0\in \b{K}$. If $n_0<\cdots < n_{m-1}$ have been determined, let $n_m\in \b{K}$ be any vertex with $n_m> n_{m-1}$, $R(n_m, n_{m-1}) = 1$, and $R(n_m, r) = 0$ for any $r< n_m$. Note that $c(n_m)|_{n_{m-1}} = 0^{n_{m-1}}$ and $c(n_m)|_{n_{m-1}+1} = (0^{n_{m-1}})^\frown 1$. In particular, the map $\pi_{n_{m-1}}\colon \{c(n_m)|_{n_{m-1}+1}\}\to \{c(n_m)|_{n_{m-1}}\}$ is not an age map. It follows that $\overline{\{n_m\}} = \{n_i: i\leq m\}$, and thus the closures of singletons can be arbitrarily large.
\end{exa}
\vspace{2 mm}

\begin{theorem}
	\label{Thm:EnvelopesImplyDegrees}
	Suppose there is $\eta\in \oemb(\b{K}, \b{K})$ which satisfies all of the following:
	\begin{enumerate}
		\item 
		For every $n < \omega$, there is $D_n< \omega$ so that every $S\subseteq \eta[\b{K}]$ with $|S| \leq n$ has $|\overline{S}|\leq D_n$,
		\item 
		For every finite $S\subseteq \eta[\b{K}]$, we have $S = \Int(\overline{S})$.
		\end{enumerate} 
	Then the ordered Ramsey degree in $\b{K}$ of any enumerated $\b{A}\in \c{K}$ is at most $$\ell := |\{\b{B}\in \c{K}: \b{B}\text{ has underlying set } d\text{ for some }d\leq D_{|\b{A}|}\}|.$$
\end{theorem}

\begin{proof}
	Fix $\b{A}\in \c{K}$ an enumerated structure, and let $\chi\colon \oemb(\b{A}, \b{K})\to r$ be a coloring for some $r< \omega$. Let $\b{B}\in \c{K}$ be any enumerated structure. We define a coloring $\xi_\b{B}\colon \aemb(\ct^\b{B}, \ct)\to r$ as follows. Fix $g\in \aemb(\ct^\b{B}, \ct)$. If there is $f\in \oemb(\b{A}, \b{K})$ with $f[ \b{A}] = \Int(\tilde{g}[\b{B}])$, we set $\xi_\b{B}(g) = \chi(f)$. If there is no such $f$, choose $\xi_\b{B}(g)$ arbitrarily. Use Theorem~\ref{Thm:Milliken} repeatedly to find $h\in \aemb(\ct, \ct)$ so that $\xi_\b{B}\cdot h$ is monochromatic, say with color $j_\b{B}< r$, for every enumerated $\b{B}\in \c{K}$ of size at most $D_{|\b{A}|}$. Then $\tilde{h}\circ \eta\in \oemb(\b{K}, \b{K})$, and we claim that $|\chi\cdot (\tilde{h}\circ \eta)[\oemb(\b{A}, \b{K})]| \leq \ell$. 
	To see this, fix $f\in \oemb(\b{A}, \b{K})$. Then by assumption, we have $|\overline{\eta\circ f[\b{A}]}|\leq D_{|\b{A}|}$ and $\Int(\overline{\eta\circ f[\b{A}]}) = \eta\circ f[\b{A}]$. Setting $S = \overline{\eta\circ f[\b{A}]}$ and $\b{B} = \b{K}_S = \b{K}_{\tilde{h}[S]}$, find $g\in \aemb(\ct^\b{B}, \ct)$ with $\tilde{g} = i_S$. Then we have $\xi_\b{B}(h\circ g) = j_\b{B} = \chi(\tilde{h}\circ \eta\circ f)$.
\end{proof}
\vspace{2 mm}

\section{Ramsey theorems for structures}
\label{Sec:RamseyStructure}

The goal of this section is to show that the assumptions of Theorem~\ref{Thm:EnvelopesImplyDegrees} hold. To do this, we need a finer analysis of how one constructs $\overline{S}$ from a finite $S\subseteq \omega$. This leads to the notion of the \emph{critical set} of $S$; denoted $\crit(S)$, these are the members of $\overline{S}$ which are ``immediately required" to be in $\overline{S}$. Crucially, we will show in Proposition~\ref{Prop:BoundedCrit} that $|\crit(S)|$ is bounded by a function of $|S|$. We then finish the section by finding $\eta\in \oemb(\b{K}, \b{K})$ where the closures of finite $S\subseteq \eta[\b{K}]$ are not that much bigger than $\crit(S)$. 

Recall that an $\c{L}$-structure $\b{F}$ is \emph{irreducible} if $R^\b{F}(a, b)\neq 0$ for every $a\neq b\in \b{F}$. Throughout this section, we will assume that $\c{K} = \mathrm{Forb}(\c{F})$, where $\c{F}$ is a \emph{finite} set of finite irreducible $\c{L}$-structures. We let 
\begin{align*}
\Irr(\c{K}) =\,\, &\{\b{A}\in \c{K}: \b{A} \text{ is enumerated and embeds into some member of } \c{F}\}\\[1 mm] 
\cup &\{\b{A}\in \c{K}: \b{A} \text{ is enumerated and } |\b{A}| = 1\}.
\end{align*}

\begin{defin}
	\label{Def:Critical}
	Let $S\subseteq \omega$ be finite.
	\begin{enumerate}
		\item 
		The \emph{splitting set} of $S$, denoted $\sp(S)$ is the set of $m< \omega$ for which $\pi_m\colon c[S]|_{m+1}\to T(m)$ is not injective.
		\item
		The \emph{age change set} of $S$, denoted $\ac(S)$, is the set of $m<\omega$ with $m\not\in S\cup \sp(S)$ and for which $\pi_m\colon c[S]|_{m+1}\to T(m)$ is not an age map.
		\item
		The \emph{critical set} of $S$ is $\crit(S):= S\cup \sp(S)\cup\ac(S)$.
	\end{enumerate}
\end{defin} 
\vspace{2 mm}

We note that $\crit(S)\subseteq \overline{S}$; one can think of $\crit(S)$ as those $m< \max(S)$ which are \emph{immediately} required to be in $\overline{S}$. In fact, $\overline{S}$ is just the smallest supserset of $S$ with $\crit(\overline{S}) = \overline{S}$.
\vspace{2 mm}

\begin{exa}
	\label{Exa:CritFor3Free}
	We revisit Example~\ref{Exa:UnboundedEnvelopes}; with $n_0<\cdots < n_m\in \b{K}$ as defined there, we have $\crit(\{n_m\}) = \{n_{m-1}, n_m\}$ and $\ac(\{n_m\}) = \{n_{m-1}\}$.
\end{exa}
\vspace{2 mm}

\begin{prop}
	\label{Prop:BoundedCrit}
	Suppose $S\subseteq \omega$ is finite. Then $|\crit(S)|$ is bounded by a function of $|S|$.
\end{prop}

\begin{proof}
	We first note that $|\sp(S)| < |S|$, so we focus on $\ac(S)$. We will build an injection $\Theta$ of $\ac(S)$ into some finite set whose size only depends on $|S|$. 
	
	Fix $m\in \ac(S)$. Choose a $c[S]|_{m+1}$-labeled $\c{L}$-structure $(\b{B}^m, \phi^m)$ with $\b{B}^m[\pi_m\circ \phi^m]\in \c{K}$, but $\b{B}^m[\phi^m]\not\in \c{K}$, and which is minimal with this property. This means that we can choose $\b{F}^m\in \c{F}$ and $g^m\in \emb(\b{F}^m, \b{B}^m[\phi^m])$ with $\b{B}^m\cup\{m\}\subseteq \im{g^m}$. Let $\b{I}^m\in \Irr(\c{K})$ and $f^m\in \oemb(\b{I}^m, \b{K})$ be such that $\im{f^m}\cup \b{B}^m = \im{g^m}$. In particular, $f^m(|\b{I}^m|-1) = m$.
	
	Define $\rho^m\colon c[S]\cap T({>}m)\to T(|\b{I}^m|)$ via $\rho^m(s)(i) = s(f^m(i))$ for $s\in c[S]\cap T({>}m)$ and $i< |\b{I}^m|$. We set $$\Theta(m) = (\b{I}^m, \rho^m)$$ 
	and argue that $\Theta$ is an injection on $\ac(S)$. Towards a contradiction, suppose $m < n\in \ac(S)$ satisfied $\Theta(m) = \Theta(n)$. Then we must have $c[S]\cap T({>}m) = c[S]\cap T({>}n)$. Considering $n$, we have that $\b{B}^n[\phi^n]\not\in \c{K}$ and $\b{B}^n[\pi_n\circ \phi^n]\in \c{K}$. So also $\b{B}^n[\pi_{m+1}\circ \phi^n]\in \c{K}$. But now consider the induced substructure of $\b{B}^n[\pi_{m+1}\circ \phi^n]$ on the set $\im{f^m}\cup \b{B}^n$; since $(\b{I}^m, \rho^m) = (\b{I}^n, \rho^n)$, we have $\im{f^m}\cup \b{B}^n\cong \b{F}^n$, a contradiction.
	
	We conclude by observing that the range of $\Theta$ has size at most
	$$\sum_{\b{I}\in \Irr(\c{K})}\sum_{j < |S|} |T(\b{I})|^{j+1}.\qedhere$$ 
\end{proof}
\vspace{2 mm}

The remainder of this section is spent showing that for $\c{K} = \mathrm{Forb}(\c{F})$, the assumptions of Theorem~\ref{Thm:EnvelopesImplyDegrees} are satisfied. When constructing the $\eta\in \oemb(\b{K}, \b{K})$ appearing in the statement of the theorem, we will control the sizes of closures of finite $S\subseteq \eta[\b{K}]$ by keeping control over $\crit(S)$. In particular, we want to ensure that the members of $\sp(S)$ and $\ac(S)$ appear in a controlled fashion in $\b{K}\setminus \eta[\b{K}]$. To actually construct $\eta$, we will first build a countable structure $\b{Y}$ with $\age{\b{Y}}\subseteq \c{K}$ which contains $\b{K}$. Even though $\b{Y}$ will not technically be an enumerated structure (its underlying set will properly contain $\omega$), we will equip $\b{Y}$ with a linear order $<_\b{Y}$ of order type $\omega$, allowing us to refer to ordered embeddings. Then we will take any $\eta\in \oemb(\b{Y}, \b{K})$ which satisfies the following lemma and simply restrict the domain to $\b{K}$. 
\vspace{2 mm}

\begin{lemma}
	\label{Lem:NiceEmbedding}
	Let $\b{Y}$ be a structure with $\age{\b{Y}}\subseteq \c{K}$, and let $<_\b{Y}$ be a linear order of $\b{Y}$ in order type $\omega$. Then there is $\eta\in \oemb(\b{Y}, \b{K})$ so that whenever $y\in \b{Y}$, $m< \eta(y)$ and $R(\eta(y), m)\neq 0$, then $m\in \eta[\b{Y}]$.
\end{lemma}

\begin{proof}
	Simply build $\eta$ inductively in $<_\b{Y}$-order, using the left density of $\b{K}$ to ensure the extra condition.
\end{proof}
\vspace{2 mm}

\begin{prop}
	\label{Prop:NiceEmbedding}
	Suppose $\langle\b{Y}, <_\b{Y}\rangle$ and $\eta\in \oemb(\b{Y}, \b{K})$ are as in Lemma~\ref{Lem:NiceEmbedding}. Suppose $n = \eta(y)$ for some $y\in \b{Y}$, and let $y_0< y_1$ be consecutive elements of $\b{Y}$ with $y_1\leq y$. Then $c(n)|_{\eta(y_1)} = \Left(c(n)|_{\eta(y_0)+1}, \eta(y_1))$.
\end{prop}

\begin{proof}
	By assumption, $R(n, m) = 0$ for any $m< n$ with $m\not\in \eta[\b{Y}]$, implying that $c(n)(m) = 0$ for all such $m$.
\end{proof}
\vspace{2 mm}

We turn towards the construction of $\b{Y}$. If $a< \omega$, let 
$$\Irr(a):= \bigsqcup_{\b{I}\in \Irr(\c{K})} \{f\in \oemb(\b{I}, \b{K}): f(|\b{I}|-1) = a\},$$ 
and write $\Irr(a) = \{\Irr(a, r): r< |\Irr(a)|\}$. We also write $|\dom{\Irr(a, r)}| = d(a, r) < \omega$. We define the underlying set of $\b{Y}$ to be 
$$\omega\sqcup \{(a, r, b): a < \omega, r < |\Irr(a)|, b< d(a, r)\}.$$
On $\omega$, $\b{Y}$ is just $\b{K}$. On $\b{Y}\setminus \omega$, we demand that for each $a< \omega$ and $r< |\Irr(a)|$, $\{(a, r, b): b< d(a, r)\}$ is an induced copy of $\dom{\Irr(a, r)}$. There are no other relations between members of $\b{Y}\setminus \omega$. Now suppose $(a, r, b)\in \b{Y}$ and $n< \omega$. We define
\begin{align*}
R^\b{Y}(n, (a, r, b)) = \begin{cases}
0\quad &\text{if } n\leq a,\\
R(n, \Irr(a, r)(b))\quad &\text{if } n> a.
\end{cases}
\end{align*}  
We define $<_\b{Y}$ to extend the usual order on $\omega$. We set $a-1 <_\b{Y} (a, r, b) <_\b{Y} a$, and $(a, r_0, b_0) <_\b{Y} (a, r_1, b_1)$ iff $r_0 < r_1$ or $(r_0 = r_1$ and $b_0 < b_1)$.

Now let $\eta\in \oemb(\b{Y}, \b{K})$ satisfy the conclusion of Lemma~\ref{Lem:NiceEmbedding}. If $s< \omega$, we define $\mathrm{Start}(s) = \max(n: c(s)|_n = 0^n)$, and if $S\subseteq \omega$, we set $\mathrm{Start}(S) = \{\mathrm{Start}(s): s\in S\}$.
\vspace{2 mm}

\begin{prop}
	\label{Prop:CritOutsideCopy}
	Suppose $S\subseteq \eta[\b{K}]$ is finite. Then any $n\in \sp(S)\cup \ac(S)\cup \mathrm{Start}(S)$ is of the form $\eta((a_n, r_n, b_n))$. 
\end{prop}

\begin{proof}
	First assume $n\in \sp(S)$. Then for some $\eta(s), \eta(t)\in S\setminus (n+1)$, $n$ is least such that $R(n, \eta(s))\neq R(n, \eta(t))$. Since $\eta$ satisfies the conclusion of Lemma~\ref{Lem:NiceEmbedding}, we must have $n = \eta(y)$ for some $y\in \b{Y}$. Suppose $a< \omega$ and $R(\eta(a), \eta(s))\neq R(\eta(a), \eta(t))$. Suppose $r< |\Irr(a)|$ is such that $\Irr(a, r)$ is the map from the singleton structure with unary $U(a)$. Then $R(\eta((a, r, 0)), \eta(s)) = R(\eta(a), \eta(s))$ and likewise for $t$. Therefore we cannot have $n = \eta(a)$ since $\eta((a, r, 0)) < \eta(a)$. A modification of this argument also works when $n = \mathrm{Start}(\eta(s))$ for some $\eta(s)\in S$.
	
	Now suppose $n\in \ac(S)$. Let $(\b{B}, \phi)$ be a $c[S]|_{n+1}$-labeled $\c{L}$-structure with $\b{B}[\pi_n\circ \phi]\in \c{K}$, but $\b{B}[\phi]\not\in \c{K}$, and assume that $(\b{B}, \phi)$ is minimal with this property. This means that for some $\b{F}\in \c{F}$, there is $g\in \emb(\b{F}, \b{B}[\phi])$ with $\b{B}\cup \{n\}\subseteq \im{g}$. Since $S\subseteq \eta[\b{K}]$ and $\eta$ satisfies the conclusion of Lemma~\ref{Lem:NiceEmbedding}, we must have $\im{g}\setminus \b{B}\subseteq \eta[\b{Y}]$, so in particular $n = \eta(y)$ for some $y\in \b{Y}$. Furthermore, we note that for some fixed $a< \omega$ and $r< |\mathrm{Irr}(a)|$, we have
	$$\im{g}\setminus (\b{B}\cup \eta[\b{K}])\subseteq \eta[\{(a, r, b): b< d(a, r)\}].$$ 
	So if $\im{g}\cap \eta[\b{K}] = \emptyset$, then $n\in \eta[\{(a, r, b): b< d(a, r)\}]$, and we are done. Towards a contradiction, suppose not.
	By the construction of $\b{Y}$, $a< m$ for every $m< \omega$ with $\eta(m)\in \im{g}$. This allows us to replace $g$ by $g_0$, where given $v\in \b{F}$, we set $g_0(v) = g(v)$ if $g(v) = \eta(m)$ for some $m< \omega$, and if $g(v) = \eta((a, r, b))$, we set $g_0(v) = \eta(\Irr(a, r)(b))$. Doing this, we may assume that $\im{g}\setminus \b{B}\subseteq \eta[\b{K}]$.   
	
	Continuing the contradiction, suppose $n = \eta(m)$ for some $m< \omega$. Find $r< |\Irr(m)|$ with $\eta[\im{\Irr(m, r)}]\cup \b{B} = \im{g}$. But now consider $\{(m, r, b): b< d(m, r)\}\subseteq \b{Y}$; for notation, set $(m, r, d(m, r)-1) = x$, and note that $\eta(x) < \eta(m) = n$. By the construction of $\b{Y}$ and since $S\subseteq \eta[\b{K}]$, we have that $\b{B}[\pi_{\eta(x)+1}\circ \phi]$ embeds $\b{F}$. More precisely, $\eta[\{(m, r, b): b< d(m, r)\}]\cup \b{B}$ is a copy of $\b{F}$, a contradiction as $\eta(x)+1 \leq n$.
\end{proof}
\vspace{2 mm}

\begin{prop}
	\label{Prop:NiceEnvelope}
	Fix $S\subseteq \eta[\b{K}]$, and set
	$$E := S\cup \eta[\{(a_n, r_n, b): b< d(a_n, r_n), n\in \sp(S)\cup\ac(S)\cup \mathrm{Start}(S)\}],$$
	where $(a_n, r_n, b_n)$ is as in Proposition~\ref{Prop:CritOutsideCopy}. Then $E$ is an envelope.
\end{prop}

\begin{proof}
	As a preliminary observation, notice that if $(a, r, b)\in \b{Y}$, then since $\eta$ satisfies the conclusion of Lemma~\ref{Lem:NiceEmbedding}, we have $\mathrm{Start}(\eta((a, r, b)))\geq \eta((a, r, 0))$.
	
	Now suppose $e_0< e_1$ are consecutive members of $E$. We need to show that \newline $\pi_{e_0+1}\colon c[E]|_{e_1}\to T(e_0+1)$ is an age map. If $e_0 = \eta((a, r, b))$ and $e_1 = \eta((a, r, b+1))$, then we are done by Proposition~\ref{Prop:NiceEmbedding}. So we may assume that either $e_1 = \eta((a, r, 0))$ or that $e_1\in S$. In either case, our preliminary observation yields that $c[E]|_{e_1}\subseteq c[S]|_{e_1}\cup \{0^{e_1}\}$. If $\pi_{e_0+1}$ is not injective on $c[E]|_{e_1}$, this implies that some member of $\sp(S)\cup \mathrm{Start}(S)$ lies between $e_0$ and $e_1$, a contradiction. If $\pi_{e_0+1}$ is injective, but not an age map on $c[E]|_{e_1}$, then it is also not an age map on $c[S]|_{e_1}$, meaning that some member of $\crit(S)$ lies between $e_0$ and $e_1$, a contradiction.
\end{proof}
\vspace{2 mm}

We can now argue that $\eta|_\b{K}$ satisfies the two assumptions of Theorem~\ref{Thm:EnvelopesImplyDegrees}. Fix a finite $S\subseteq \eta[ \b{K}]$. The envelope $E$ constructed in Proposition~\ref{Prop:NiceEnvelope} has size which is bounded by a function of $|S|$, so the first assumption holds. For the second, we note that $E\setminus S\subseteq \eta[\b{Y}]\setminus \eta[\b{K}]$. So also $\overline{S}\setminus S\subseteq \eta[ \b{Y}]\setminus \eta[\b{K}]$. Since this is true for \emph{any} finite $S\subseteq \eta[\b{K}]$, we have $\overline{(S\setminus \{s\})}\setminus (S\setminus \{s\})\subseteq \eta[\b{Y}]\setminus \eta[ \b{K}]$ for each $s\in S$. In particular, $s\not\in \overline{S\setminus \{s\}}$.

\appendix

\section{Appendix: Forcing}

This appendix provides a self-contained introduction to the ideas from forcing needed in the proof of Theorem~\ref{Thm:HL}. 

Let $\langle\bb{P}, \leq_\bb{P}\rangle$ be a poset. Given $p, q\in \bb{P}$ with  $q\leq_\bb{P} p$, we will say that $q$ \emph{extends} $p$ or that $q$ \emph{strengthens} $p$.
\vspace{2 mm}

\begin{defin}
	\label{Def:NameSubset}
	A \emph{$\bb{P}$-name for a subset of $\omega$} is any subset of $\bb{P}\times \omega$. 
\end{defin}
\vspace{2 mm}

Typically, $\bb{P}$-names are denoted by symbols with dots over them. 
\vspace{2 mm}

\begin{defin}
	\label{Def:ForcingSetNames}
	Fix a $\bb{P}$-name $\dot{L}\subseteq \bb{P}\times \omega$. Fix $p\in \bb{P}$.
	\begin{enumerate}
		\item 
		Given $n< \omega$, we write $p\Vdash ``n\in \dot{L}"$ and say ``$p$ forces that $n\in \dot{L}$" if for any $q\leq_\bb{P} p$, there is some $r\leq_\bb{P} q$ and some $r_0\geq_\bb{P} r$ with $(r_0, n)\in \dot{L}$.
		\item 
		Suppose $\dot{L}_0\subseteq \bb{P}\times \omega$ is another name. Then $p \Vdash ``\dot{L}_0\subseteq \dot{L}"$ if for any $q\leq_\bb{P} p$ and any $n< \omega$, we have $q\Vdash ``n\in \dot{L}_0"$ implies $q\Vdash ``n\in \dot{L}."$ A similar definition applies to $p\Vdash ``\dot{L}_0 = \dot{L}."$
		\item 
		$p\Vdash ``n\not\in \dot{L}"$ if there is no $q\leq_\bb{P} p $ with $q\Vdash ``n\in \dot{L}."$
		\item
		$p\Vdash ``\dot{L} \text{ is infinite}"$ if for any $q\leq_\bb{P} p$ and $m< \omega$, there are $r\leq_\bb{P} q$ and $n > m$ with $r\Vdash ``n\in \dot{L}."$
		\item 
		Suppose $\dot{L}_0,...,\dot{L}_{k-1}\subseteq \bb{P}\times \omega$ are names. Then $p\Vdash ``\bigcap_{i< k} \dot{L}_i \neq \emptyset"$ if for any $q\leq_\bb{P} p$, there are $r\leq_\bb{P} q$ and $n< \omega$ so that $r\Vdash ``n\in \dot{L}_i"$ for each $i< k$. 
		
		We have $p\Vdash ``\bigcap_{i< k} \dot{L}_i \text{ is infinite}"$ if for any $q\leq_\bb{P} p$ and any $m< \omega$, there are $r\leq_\bb{P} q$ and $n> m$ so that $r\Vdash ``n\in \dot{L}_i"$ for each $i< k$.
	\end{enumerate}
\end{defin}
\vspace{2 mm}

\begin{rem}
	In general, one can define the \emph{forcing relation} $p\Vdash \phi$ much more generally. Here and below, we choose to be much more explicit, defining $p \Vdash \phi$ by hand for only those $\phi$ that we will need. One general remark is worth mentioning: if $p\Vdash \phi$ and $q\leq_\bb{P} p$, then also $q\Vdash \phi$.
\end{rem}
\vspace{2 mm}

\begin{defin}
	\label{Def:ForcingPowersetNames}
	A \emph{$\bb{P}$-name for a collection of subsets of $\omega$} is any subset of $\bb{P}\times \c{P}(\bb{P}\times \omega)$. In particular, if $\dot{\c{F}}$ is a $\bb{P}$-name for a collection of subsets of $\omega$ and $(p, \dot{L})\in \dot{\c{F}}$, then $\dot{L}$ is a $\bb{P}$-name for a subset of $\omega$.
	
	Suppose $\dot{\c{F}}$ is a $\bb{P}$-name for a collection of subsets of $\omega$ and that $\dot{L}$ is a $\bb{P}$-name for a subset of $\omega$. Then given $p\in \bb{P}$, we have $p\Vdash ``\dot{L}\in \dot{\c{F}}"$ if for any $q\leq_\bb{P} p$, there are $r\leq_\bb{P} q$, $r_0\geq_\bb{P} r$,  and $\dot{L}_0\subseteq \bb{P}\times \omega$ with $(r_0, \dot{L}_0)\in \dot{\c{F}}$ and $r\Vdash ``\dot{L}_0 = \dot{L}."$
\end{defin}
\vspace{2 mm}

\begin{defin}
	\label{Def:ForcingFilterNames}
	Suppose $\dot{\c{F}}$ is a $\bb{P}$-name for a collection of subsets of $\omega$. We say that $\dot{\c{F}}$ has the \emph{finite intersection property}, or FIP, if for any $(p_i, \dot{L}_i)\in \dot{\c{F}}$ and any $q\in \bb{P}$ with $q\leq_\bb{P} p_i$ for each $i< k$, we have $q\Vdash ``\bigcap_{i< k} \dot{L}_i\neq \emptyset."$ We say that $\dot{\c{F}}$ has the \emph{strong FIP}, or SFIP, if in the above situation, we have $q\Vdash ``\bigcap_{i< k} \dot{L}_i \text{ is infinite}."$
	
	A name $\dot{\c{U}}\subseteq \bb{P}\times \c{P}(\bb{P}\times \omega)$ is a \emph{$\bb{P}$-name for an ultrafilter on $\omega$} if $\dot{\c{U}}$ is maximal with respect to having the FIP. If $\dot{\c{U}}$ also has the SFIP, we say that $\dot{\c{U}}$ is \emph{non-principal}.
\end{defin}
\vspace{2 mm}

\begin{prop}
	\label{Prop:ForcingUltExt}
	Suppose $\dot{\c{F}}$ is a $\bb{P}$-name for a collection of subsets of $\omega$ which has the FIP. Then there is a $\bb{P}$-name $\c{U}$ for an ultrafilter on $\omega$ with $\dot{\c{F}}\subseteq \dot{\c{U}}$. If $\dot{\c{F}}$ has the SFIP, then $\dot{\c{U}}$ can be chosen to be non-principal.
\end{prop}

\begin{proof}
	The first claim follows from Zorn's lemma. For the second, suppose $\dot{\c{F}}$ has the SFIP. Set 
	$$\dot{\c{F}}_1 = \dot{\c{F}} \cup \{(p, \bb{P}\times (\omega\setminus n)): p\in \bb{P}, n< \omega\}.$$
	Then since $\dot{\c{F}}$ has the SFIP, $\dot{\c{F}}_1$ has the FIP, so let $\dot{\c{U}}\supseteq \dot{\c{F}}$ be a $\bb{P}$-name for an ultrafilter on $\omega$. Since $(p, \bb{P}\times (\omega\setminus n))\in \dot{\c{U}}$ for every $p\in \bb{P}$ and $n<\omega$, it follows that $\dot{\c{U}}$ has the SFIP.
\end{proof}
\vspace{2 mm}

We end by collecting some basic facts about names for ultrafilters on $\omega$, all of which are consequences of the fact that names for ultrafilters are maximal with respect to having the FIP.
\vspace{2 mm}

\begin{fact}
	\label{Prop:ForcingBasicUlts}
	Let $\dot{\c{U}}$ be a $\bb{P}$-name for an ultrafilter on $\omega$.
	\begin{enumerate}
		\item 
		If $p\in \bb{P}$ and $\dot{L}\subseteq \bb{P}\times \omega$, then $p\Vdash ``\dot{L}\in \dot{\c{U}}"$ iff $(p, \dot{L})\in \dot{\c{U}}$.
		\item 
		If $p\in \bb{P}$ and $\dot{L}_0, \dot{L}_1\subseteq \bb{P}\times \omega$ with $(p, \dot{L})\in \dot{\c{U}}$ and $p\Vdash ``\dot{L}_0\subseteq \dot{L}_1,"$ then $(p, \dot{L}_1)\in \dot{\c{U}}$.
		\item 
		Suppose $\dot{L}\subseteq \bb{P}\times \omega$ and $\dot{L} = \dot{L}_0\cup \dot{L}_1$. If $(p, \dot{L})\in \dot{\c{U}}$, then for some $q\leq_\bb{P} p$ and $i< 2$, we have $(q, \dot{L}_i)\in \dot{\c{U}}$.
	\end{enumerate}
\end{fact}

\vspace{15 mm}

\noindent
Andy Zucker

\noindent
Universit\'{e} Claude Bernard - Lyon 1

\noindent
zucker@math.univ-lyon1.fr

\end{document}